\documentclass{article}

\usepackage{amsmath,amssymb,amsthm}
\usepackage{framed}
\usepackage{mathrsfs}
\usepackage{graphicx}
\usepackage{geometry}
\usepackage[small,bf]{caption}

\newcommand\onto{\!\rightarrow\!}

\newcommand{\CC}{A}
\newcommand{\const}{\mathscr{C}}
\newcommand{\diam}{\mathrm{diam}}
\newcommand\reals{\mathbb R}
\newcommand\IR{\mathbb R}
\newcommand\manifold{\Gamma}
\newcommand\neumann{\mathcal{N}}
\newcommand\dirichlet{\mathcal{D}}
\newcommand\domain\Omega
\newcommand\slip{b}
\newcommand\jump[1]{[\hspace{-1pt}[\hspace{1pt}#1\hspace{1pt}]\hspace{-1pt}]}
\newcommand\mean[1]{\{#1\}}
\newcommand\divstress{\mathrm{div}\,\sigma}
\newcommand\grad{\nabla}
\newcommand{\ucA}{\underline{c}{}_A}
\newcommand{\bcA}{\overline{c}{}_A}
\newcommand{\vctr}[1]{\boldsymbol{#1}}
\newcommand{\HH}{\vctr{H}}
\newcommand{\TT}{\vctr{T}}

\newcommand{\VV}{\vctr{V}}
\newcommand{\LL}{\vctr{L}}

\newcommand{\vp}{\varphi}
\newcommand{\zvp}{z_{\vp}}
\newcommand{\supp}{\mathrm{supp}}
\newcommand{\meas}{\mathrm{meas}}

\newcommand{\ab}{a_{\manifold}}

\newcommand{\FIG}[1]{\ref{fig:#1}}
\newcommand{\lift}[1]{\ell_{#1}}
\newcommand{\EQ}[1]{(\ref{eqn:#1})}
\newcommand{\subdivstress}{\mathrm{div}\sigma}

\DeclareMathOperator*{\arginf}{arg\:inf}
\DeclareMathOperator*{\argmin}{arg\:min}

\newtheorem{theorem}{Theorem}
\newtheorem{lemma}[theorem]{Lemma}
\newtheorem{corollary}[theorem]{Corollary}

\newtheorem{conjecture}{Conjecture}

\title{Discontinuities without discontinuity:\\The Weakly-enforced Slip Method}

\author{%
  G.J.\ van Zwieten$^{1,3}$\footnote{email: g.j.v.zwieten@tue.nl},
  E.H.\ van Brummelen$^{1,2}$,
  K.G.\ van der Zee${}^{1}$, \\
  M.A.\ Guti\'errez${}^{3}$,
  R.F.\ Hanssen${}^{4}$ \\
  \\
  \normalsize
  \begin{tabular}{cl}
  $^1$ & Eindhoven University of Technology, \\
       & Department of Mechanical Engineering, \\
       & P.O. Box 513, 5600 MB Eindhoven, The Netherlands \\[1ex]
  $^2$ & Eindhoven University of Technology, \\
       & Department of Mathematics \& Computer Science, \\
       & P.O. Box 513, 5600 MB Eindhoven, The Netherlands \\[1ex]
  $^3$ & Delft University of Technology, \\
       & Department of Mechanical, Maritime and Materials Engineering, \\
       & P.O. Box 5058, 2600 GB Delft, The Netherlands \\[1ex]
  $^4$ & Delft University of Technology, \\
       & Department of Civil Engineering and Geosciences, \\
       & P.O. Box 5048, 2600 GA Delft, The Netherlands
  \end{tabular}}

\begin{document}

\maketitle

\begin{abstract}
  Tectonic faults are commonly modelled as Volterra or Somigliana dislocations
  in an elastic medium. Various solution methods exist for this problem.
  However, the methods used in practice are often limiting, motivated by
  reasons of computational efficiency rather than geophysical accuracy. A
  typical geophysical application involves inverse problems for which many
  different fault configurations need to be examined, each adding to the
  computational load. In practice, this precludes conventional finite-element
  methods, which suffer a large computational overhead on account of geometric
  changes. This paper presents a new non-conforming finite-element method based
  on weak imposition of the displacement discontinuity. The weak imposition of
  the discontinuity enables the application of approximation spaces that are
  independent of the dislocation geometry, thus enabling optimal reuse of
  computational components. Such reuse of computational components renders
  finite-element modeling a viable option for inverse problems in geophysical
  applications. A detailed analysis of the approximation properties of the new
  formulation is provided. The analysis is supported by numerical experiments
  in~2D and~3D. \\

  \noindent\textbf{Keywords:} Volterra dislocation, Finite Element Method, weak
  imposition, linear elasticity, tectonophysics.
\end{abstract}

\section{Introduction}

The world is perpetually reminded of the fact that seismic hazard is still
beyond reach of prediction --- as it was most recently by the disaster that
struck Japan. The difficulty is not just to predict the exact moment of
failure, which, as argued by some~\cite{geller97}, might never reach a level of
practicality. It is also the nature of the risk, and the extent to which stress
is accumulating, that turns out to be surprisingly difficult to constrain. The
2011 Tohoku-Oki earthquake demonstrated a great lack of understanding of
ongoing tectonics \cite{kerr11}. Arguably, a better understanding could have
reduced the secondary effects if such information had led to more apt measures
and regulations.

The main reason for this poor state of information can be
traced to the absence of direct measurements. The primary quantities of
interest, being the magnitude and orientation of the stress tensor in the
earth's crust, can be obtained only through tedious, expensive, point-wise
measurements. A viable broad scale method to directly measure the global stress
field does not exist. For this reason information is obtained mostly from
secondary observables, earthquakes themselves being an important source.
Earthquakes represent significant, near instantaneous changes in the global
stress field. By accurately determining the location of the segment of the
fault that collapsed, one can progressively update the stress field and evolve
it in time. This way the tectonic evolution is monitored, and hazardous areas
can be identified as regions where stress accumulates. For successful tracking
of stress development, however, it is essential to understand the tectonic
mechanism behind every earthquake. This includes the location and geometry of
the section of the fault that collapsed, and the direction and magnitude of
fault slip. It is increasingly popular to base such analyses on local
co-seismic surface displacements. This type of information has become available
since the nineties with the advent of space borne interferometric SAR
measurements of the earth's surface, and with the widespread availability of
GPS measurements \cite{tronin06}. Analysis of this data has in recent years
seen rapid adoption and is now routinely performed for all major earthquakes.

A mechanical model is required to connect observations to physics. Most
commonly (if not exclusively) used is an elastic dislocation model, based on
the assumption that on short time scales, nonlinear (plastic) effects are
negligible. The model embeds a displacement discontinuity of given location and
magnitude in an elastic me\-di\-um, causing the entire medium to deform under
the locked-in stress. Many different solution methods have been developed for
this particular problem, based on analytical solutions or numerical
approximations; see for instance \cite{gertjan13} for
an overview of the most prominent methods. However, methods founded on
analytical solutions generally dictate severe model simplifications, such as
elastic homogeneity or generic geometries, which restricts their validity. The
computational complexity of methods based on numerical approximation, on the
other hand, is typically prohibitive in practical applications. Because in
practice the surface displacements are given, and the dislocation parameters
are the unknowns, the computational setting is always that of an inverse
problem. A typical inversion requires several thousands of evaluations of the
forward model, and therefore computational efficiency is a key requirement.
Moreover, the forward problems in the inversion process are essentially
identical, except for the fault geometry. Reuse of computational components,
such as approximate factors of the system matrix, is imperative for efficiency
of the inversion. Current numerical methods for seismic problems do not offer
such reuse options.

Finite-element methods provide a class of numerical techniques that
are particularly versatile in terms of modeling capabilities in
geophysics. Finite-element methods allow
for elastic heterogeneity, anisotropy, and topography; all things that can
not well be accounted for with currently used analytical and semi-analytical
methods. In geophysical practice, finite-element methods are however often rejected 
for reasons of computational cost. The high computational cost can be
retraced to the condition that the geometry of the fault 
coincides with element edges, which is a requirement engendered by the
strong enforcement of the dislocation; see~\cite{melosh81}.
Consequently, the mesh geometry depends on the fault, which in turn
implies that mesh-dependent components such as the stiffness matrix and
approximate factorizations of that matrix cannot be reused for
different fault geometries and must be recomputed whenever the
geometry of the fault changes. The recomputation of these components
in each step of a nonlinear inversion process leads to a prohibitive
overall computational complexity.

To overcome the complications of standard finite-element techniques in
nonlinear inversion processes in tectonophysics, this paper introduces
the \emph{Weakly-enforced Slip Method\/} (WSM), a new numerical method in 
which displacement discontinuities are weakly imposed. The WSM
formulation is similar to Nitsche's variational principle for
enforcing Dirichlet boundary conditions~\cite{Nitsche:1971fk}.
The weak imposition of the discontinuity in WSM decouples the finite element
mesh from the geometry of the fault, which
renders the stiffness matrix and derived objects such as approximate factors
independent of the fault and enables reuse of these objects. Therefore, even 
though the computational work required for a single realisation of the 
fault geometry is comparable to that of standard FEM, reuse of components
makes WSM significantly more efficient when many different fault
geometries are considered. This makes finite-element computations based
on WSM a viable option for nonlinear inverse problems.

A characteristic feature of WSM is that it employs 
standard continuous finite-element approximation spaces, as opposed to the
conventional FEM split-node approach~\cite{melosh81} which introduces
actual discontinuities in the approximation space. Instead, WSM
approximations feature a `smeared out' jump with sharply localized
gradients. We will establish that the error in the WSM approximation
converges only as $O(h^{1/2})$ in the~$\LL^2$\nobreakdash-norm as the 
mesh width~$h$ tends to zero and that the error diverges 
as~$O(h^{-1/2})$ in the energy norm. In addition, however, we will
show that the WSM approximation displays optimal local convergence
 in the energy norm, i.e., optimal convergence rates are obtained on 
any subdomain excluding a neighborhood of the dislocation. The numerical
experiments convey that WSM also displays optimal local convergence in the
$\LL^2$\nobreakdash-norm.

The remainder of this manuscript is organized as follows.
Section~\ref{sec:formulation} presents strong and weak formulations of Volterra's 
dislocation problem, and derives the corresponding 
lift-based finite-element formulation. Section~\ref{sec:wsm} introduces the
Weakly-enforced Slip Method based on two formal derivations, viz., by collapsing the 
support of the lift onto the fault and by application of Nitsche's variational principle.
In Section~\ref{sec:convergence}, we examine the approximation properties of the
WSM formulation. Section~\ref{sec:results} verifies and illustrates the approximation properties
on the basis of numerical experiments for several~2D and~3D test cases. In addition, to 
illustrate the generality of WSM, in the numerical experiments we consider several test cases 
that violate the conditions underlying the error estimates in Section~\ref{sec:convergence},
such as discontinuous slip distributions and rupturing dislocations. Section~\ref{sec:concl}
presents concluding remarks.

\section{Problem formulation}
\label{sec:formulation}

In this section we define Volterra's dislocation problem. We will postulate the strong formulation in
\ref{sec:strongform}, followed by a derivation of the weak formulation in
\ref{sec:weakform}. The latter will serve as a basis for the derivation of the
Finite Element approximations, for which we lay foundations in
Section~\ref{sec:fem}.

\subsection{The strong form}
\label{sec:strongform}

We start by defining the geometric setup. We consider an open bounded domain $
\domain \subset \reals^N $ ($N=2,3$) with Lipschitz boundary $\partial\domain$.
An $(N\!-\!1)$-dimensional Lipschitz manifold $ \manifold $, referred as the fault,
divides the domain in two disjoint open subdomains $ \domain^+ $ and $
\domain^- $, such that $ \domain = \mathrm{int}(\overline{\domain^+} \cup
\overline{\domain^-})$. We equip $\manifold$ with a unit normal vector $
\nu:\manifold\to\reals^N $ directed into the subdomain $ \domain^- $. The fault
supports a slip distribution~$\slip: \Gamma \rightarrow \mathbb R^N$,
corresponding to a dislocation. The fault is referred to as a non-rupturing
fault if the dislocation $ \varkappa = \supp(\slip)$ is compact in~$\Gamma$,
and a rupturing fault otherwise. Let us note that in tectonics, rupturing
faults correspond to intersections of the slip plane with the surface of the
earth. Figure~\ref{fig:domain} illustrates this setup for $ N = 2 $. It is to
be noted that $ b $ need not be tangential to the fault. If $ b $ has a
non-vanishing normal component then the fault is opening, such as may be caused
by an intruding material.

\begin{figure}
  \centering
  \includegraphics{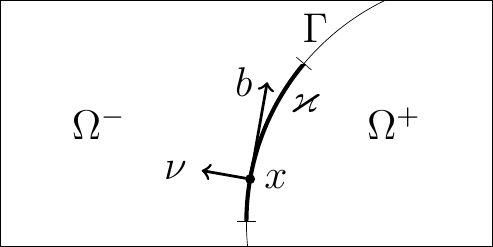}
  \caption{Definition of the domain $ \Omega $, normal vector $ \nu $, slip
    vector $ b $, fault $ \Gamma $ and dislocation $ \varkappa $.}
  \label{fig:domain}
\end{figure}

The displacement field generated by the dislocation is represented by $ u :
\domain \setminus \manifold \onto \reals^N $. For convenience, we restrict
our considerations to linear elasticity.
We denote by the map $u\mapsto\epsilon(u)$ the strain tensor corresponding to
the displacement field~$u$, according~to
\begin{equation} 
  \label{eqn:strain} 
  \epsilon( u ) := \tfrac 1 2 \big[ \nabla u + ( \nabla u )^T \big], 
\end{equation}
under the assumption that $ u $ is differentiable on $ \domain \setminus \manifold $.
The constitutive behavior corresponds to Hooke's law:
\begin{equation} 
  \label{eqn:stressstrain} 
  \sigma( u ) := \CC:\epsilon( u),
\end{equation} 
i.e., $\sigma_{ij}(u)=A_{ijkl}\epsilon_{kl}(u)$, where we adhere to the
convention on summation on repeated indices. The tensors $\sigma$ and $\CC$ are
referred to as the stress tensor and the elasticity tensor, respectively. The
elasticity tensor is subject to the usual symmetries $ \CC_{ijkl} = \CC_{ijlk}
= \CC_{klij} $. Moreover, we assume that it is bounded and satisfies a strong
positivity condition, i.e., there exist positive constants $\bcA>0$
and~$\ucA>0$ such that:
\begin{equation} 
  \label{eqn:Abp}
  \ucA{}\,e_{ij}e_{ij}
  \leq
  \CC_{ijkl}e_{ij}e_{kl}\leq\bcA{}e_{ij}e_{ij}
\end{equation}
for all tensors $e$.
The elasticity tensor is in principle allowed to vary over the 
domain~$\domain$, subject to the above conditions. Auxiliary smoothness conditions on~$\CC$
will be introduced later.

To facilitate the formulation, we denote by $\sigma_n(u):=\sigma(u)\cdot{}n$
the traction on a boundary corresponding to the displacement field~$u$.
Moreover, we introduce the jump operator~$\jump{\cdot}$ and average operator
$\mean{\cdot}$ according to:
\begin{subequations}
  \begin{align}
    \jump{v}& : \manifold \ni x \mapsto v^+( x ) - v^-( x ), \label{eqn:jump} \\
    \mean{v}& : \manifold \ni x \mapsto \tfrac{1}{2}\big[v^+( x ) + v^-( x )\big], \label{eqn:av}
  \end{align}
\end{subequations}
where $v^+$ and $v^-$ represent the traces of $v$ from within~$\domain^+$ and~$\domain^-$, respectively.
Given a partition of the boundary $\partial\Omega$ into $\dirichlet\neq\emptyset$ and $\neumann$
such that $\dirichlet\cap\neumann=\emptyset$, we define the Volterra dislocation problem as follows:

\begin{framed}
  \noindent \textbf{Strong formulation:} given a body force $ f:\Omega\to\reals^N $, 
  a displacement $g:\dirichlet\to\reals^N$ and a traction $h:\neumann\to\reals^N$,
  and slip $ \slip:\manifold\to\reals^N $, find displacement field
  $ u:\domain\setminus\manifold\to\reals^N $ such that
  \begin{subequations}
    \label{eqn:BVP}
    \begin{align}
      -\divstress(u) &= f \text{ in } \domain \setminus
      \manifold \label{eqn:strong_equilibrium} \\
      \jump{u} &= \slip \text{ on } \manifold \label{eqn:strong_dispjump} \\
      \jump{ \sigma_{\nu}(u) } &= 0 \text{ on } \manifold \label{eqn:strong_tracjump} \\
      u &= g \text{ on } \dirichlet \label{eqn:strong_farfield}\\
      \sigma_n( u ) &= h \text{ on } \neumann \label{eqn:strong_surface} 
    \end{align}
  \end{subequations}
\end{framed}
\noindent%
Equation~\eqref{eqn:strong_equilibrium} is the usual equilibrium condition,
which applies everywhere in $\domain$ except on the manifold~$ \manifold$.
Equations~\eqref{eqn:strong_dispjump} and~\eqref{eqn:strong_tracjump}
respectively express that the displacements at the boundaries of $ \domain^+ $
and $ \domain^- $ differ by the slip vector~$ b $, and that the tractions at
the boundaries of $ \domain^+ $ and $ \domain^- $ are in static equilibrium,
i.e., equal and opposite. It is to be remarked that this condition corresponds
to a standard linear approximation in the small-slip limit, as the traction
equilibrium at the fault occurs in fact in the deformed configuration. The
boundary conditions ~\eqref{eqn:strong_farfield} and~\eqref{eqn:strong_surface}
correspond to Dirichlet and Neumann boundary conditions, representing a prescribed displacement 
and a prescribed traction, respectively.
 
To facilitate the presentation, we note that the solution to the Volterra dislocation problem~\eqref{eqn:BVP}
can be separated into a discontinuous part $u_0:\domain\setminus\manifold\to\reals^N$ with homogeneous data
and a continuous part $u_1:\Omega\to\reals^N$ with inhomogeneous data: 
\begin{equation}
\label{eqn:strongform}
\left.
\begin{aligned}
-\divstress(u_0) &= 0\quad\text{in } \domain \setminus \manifold 
\\
u_0&=0\quad\text{on } \dirichlet
\\
\sigma_n( u_0 ) &= 0 \quad\text{on } \neumann
\\
\jump{u_0} &= \slip\quad \text{ on } \manifold 
\\
\jump{ \sigma_{\nu}(u_0) } &= 0\quad\text{ on } \manifold 
\end{aligned}
\right.
\qquad\qquad
\left.
\begin{aligned}
-\divstress(u_1) &= f\quad\text{in } \domain
\\
u_1&=g\quad\text{on } \dirichlet
\\
\sigma_n( u_1 ) &=h \quad\text{on } \neumann
\\
{}&{}\\
{}&{} 
\end{aligned}
\right.
\end{equation}

The sum $u_0+u_1$ satisfies~\eqref{eqn:BVP}. Therefore, the inhomogeneous data $f,g,h$ in~\eqref{eqn:BVP} can be 
treated separately in a standard continuous elasticity problem, and without loss of generality we can restrict 
our consideration to homogeneous data. We retain $f:=0$ to identify the right member of~\eqref{eqn:strong_equilibrium}.

For rupturing faults, some compatibility conditions arise with respect to the
boundary conditions. In particular, an intersection of the dislocation with the
boundary of the domain is only admissible at the Neumann boundary~$\neumann$.
Otherwise, an inadmissible incompatibility between the jump
condition~$\jump{u_0}=b$ and the homogeneous Dirichlet condition~$u_0=0$
ensues.

\subsection{The weak form}
\label{sec:weakform}

To derive the weak formulation of~\eqref{eqn:BVP}, we note that for any piecewise smooth 
function $v$ from $\domain\setminus\manifold$ into $\reals^N$, we have the identities:
\begin{align}
  \label{eqn:varform1}
  \int_{\domain \setminus \manifold} v \cdot \divstress(u)
  &= \oint_{\partial\domain^-}v\cdot\sigma_n(u)
  - \int_{\domain^-} \sigma( u ):\nabla v \nonumber\\
  &\quad + \oint_{\partial\domain^+}v\cdot\sigma_n(u)
  - \int_{\domain^+} \sigma( u ):\nabla v 
  \nonumber\\
  &= \oint_{\partial\Omega}v\cdot\sigma_n(u)
  - \int_{\Omega\setminus\manifold} \sigma( u ):\nabla v \nonumber\\
  &\quad + \int_{\manifold} \jump{v}\cdot\mean{\sigma_{\nu}(u)}
  + \mean{v}\cdot\jump{\sigma_{\nu}(u)}
\end{align}
The first identity results from integration by parts. The second identity 
follows from a rearrangement of the boundary terms and
\begin{align}
  \label{eqn:rearrangement}
  \int_{\manifold}
  v^+\cdot\sigma_n^+(u)+v^-\cdot\sigma_n^-(u)
  &= \int_{\manifold} v^+\cdot\sigma_{\nu}^+(u)-v^-\cdot\sigma_{\nu}^-(u)
  \nonumber\\
  &= \int_{\manifold} \big(v^+-v^-\big)\,\tfrac{1}{2}\big(\sigma_{\nu}^+(u)+\sigma_{\nu}^-(u)\big)
  \nonumber\\
  &\quad + \int_{\manifold} \tfrac{1}{2}\big(v^++v^-\big)\,\big(\sigma_{\nu}^+(u)-\sigma_{\nu}^-(u)\big)
\end{align}
In the weak formulation, the admissible displacement fields will be insufficiently regular
to ensure the existence of the tractions $\sigma_n(\cdot)$. Hence, the terms involving these
tractions in~\eqref{eqn:varform1} must be eliminated by means of the boundary conditions and auxiliary conditions on~$v$.
The traction term on the Neumann boundary can be eliminated by means of~\eqref{eqn:strong_surface}.
To remove the traction term on~$\partial\Omega$, we stipulate that~$v$ vanishes on~$\dirichlet$. 
The traction average in the final term of~\eqref{eqn:varform1} is eliminated by requiring that~$v$ be continuous.
The traction jump in the final term is deleted by means of the traction-continuity condition~\eqref{eqn:strong_tracjump}.

Summarizing, we find that a solution~$u$ of~\eqref{eqn:BVP} satisfies
\begin{equation}
  \label{eqn:varid1}
  \ab(u,v)=l(v)
\end{equation}
for all sufficiently smooth functions $v:\domain\to\reals^N$ that vanish on~$\dirichlet$,
where
\begin{subequations}
  \label{eqn:ab}
  \begin{align}
  \ab(u,v)&=
  \int_{\domain\setminus\manifold}  \sigma( u ):\nabla v
  \label{eqn:aba}\\
  l(v)
  &=
  \int_{\Omega}v\cdot{}f
  \label{eqn:abb}
  \end{align}
\end{subequations}
Note that $v$ is assumed to be smooth on $\domain$ and, in particular, that it is continuous
across the fault $\manifold$.

To furnish a functional setting for the weak formulation of~\eqref{eqn:BVP} based on~\eqref{eqn:varid1}, 
we denote by $\HH{}^k(\Omega)$ the usual Sobolev space of square-integrable functions 
from~$\Omega$ into $\reals^N$ with square-integrable distributional derivatives of order $\leq{}k$, 
equipped with the inner product 
\begin{equation*}
  (u,v)_{k,\Omega}=\sum_{|\alpha|\leq{}k}\int_{\Omega}D^{\alpha}u\cdot{}D^{\alpha}v
\end{equation*}
and the corresponding norm $\smash[b]{\|\cdot\|_{k,\Omega}}$ and semi-norm $\smash[b]{|\cdot|_{k,\Omega}}$.
For the square-integrable functions and the corresponding norm and inner-product, we introduce the
condensed notation $\LL^2(\Omega):=\HH{}^0(\Omega)$, $\|\cdot\|_{\Omega}:=\|\cdot\|_{0,\Omega}$ 
and~$(\cdot,\cdot)_{\Omega}:=(\cdot,\cdot)_{0,\Omega}$.
We denote by $\HH{}^1_{0,\dirichlet}(\Omega)$ the subspace of $\HH{}^1(\Omega)$ of functions
that vanish on $\dirichlet\subseteq\partial\Omega$. To accommodate the discontinuity 
corresponding to the dislocation, we introduce the lift operator $\lift{(\cdot)}$, which assigns to 
any suitable slip $b:\manifold\to\reals^N$ a function
$\lift{b}$ in $\HH{}^1_{0,\dirichlet}(\domain\setminus\manifold)$ such that
$\jump{\lift{b}}=b$. A precise specification of conditions on the slip
distribution is given in Section~\ref{sec:convergence}.
The weak formulation of~\eqref{eqn:BVP} based on~\eqref{eqn:varid1} writes
\begin{framed}
\noindent \textbf{Weak formulation:} given the lift~$ \lift{b}\in{}\HH{}^1_{0,\dirichlet}(\domain\setminus\manifold) $, 
find $u \in \lift{b} + \HH{}^1_{0,\dirichlet}(\domain) $ such that
\begin{equation}
  \label{eqn:weak_varform}
  \ab(u,v) = l(v)
  \quad
  \forall{}v\in{}\HH{}^1_{0,\dirichlet}(\Omega)\,.
\end{equation}
\end{framed}
\noindent
The bilinear form $\ab:\HH{}^1(\Omega\setminus\manifold)\times
\HH{}^1(\Omega\setminus\manifold)\to\reals$ and linear form
$l:\HH{}^1(\Omega\setminus\manifold)\to\reals$ in~\eqref{eqn:weak_varform}, are
the extensions (by continuity) of the corresponding forms in~\eqref{eqn:ab}.
The treatment of the dislocation by means of a lift operator
in~\eqref{eqn:weak_varform} is analogous to the treatment of inhomogeneous
Dirichlet boundary conditions in weak formulations; see, e.g., \cite[p.113
]{ErnGuermond2004}. Let us note that in the weak
formulation~\eqref{eqn:weak_varform}, we have identified
$\{u\in{}\HH{}^1(\domain\setminus\manifold):\jump{u}=0\}$
with~$\HH^1(\domain)$. This identification is unambiguous, on account of a
one-to-one correspondence between the functions in these spaces.

To analyze the weak formulation~\eqref{eqn:weak_varform}, and to prepare the
presentation of the Weakly-enforced-Slip method in Section~\ref{sec:wsm}, we
note that~\eqref{eqn:weak_varform} is to be interpreted in the following
manner: find $u:=\bar{u}+\lift{b}$ with $\bar{u}\in
\HH{}^1_{0,\dirichlet}(\domain) $ such that  
\begin{equation}
  \label{eqn:weak_varform1}
  \qquad{}
  a(\bar{u},v) = l(v)-\ab(\lift{b},v)
  \qquad
  \forall{}v\in{}\HH{}^1_{0,\dirichlet}(\Omega)\,.
\end{equation}
where the bilinear form $a:\HH{}^1(\domain)\times\HH{}^1(\domain)\to\reals$,
\begin{equation}
  \label{eqn:a_restrict}
  a(u,v)=\int_{\domain}\sigma(u):\nabla{}v,
\end{equation}
corresponds to the restriction of $\ab(\cdot,\cdot)$ to $\HH{}^1(\domain)\times{}\HH{}^1(\domain)$.
Indeed, for all pairs $(u,v)\in{}\HH{}^1(\domain)\times{}\HH{}^1(\Omega)$, the function \mbox{$\sigma(u):\nabla{}v$} is 
Lebesgue integrable on~$\domain$, and because the manifold $\manifold$ corresponds to a set of $N$-Lebesgue measure zero, 
the integrals of \mbox{$\sigma(u):\nabla{}v$} on $\domain\setminus\manifold$ and on~$\domain$ coincide. 
It is important to note that the restriction of the bilinear form~$\ab(\cdot,\cdot)$ to $\HH{}^1(\domain)\times{}\HH{}^1(\domain)$ 
in~\eqref{eqn:a_restrict} is independent of the fault~$\manifold$. The function $\bar{u}$ in~\eqref{eqn:weak_varform}
is referred to as the continuous complement of the solution~$u$ with respect to the jump lift~$\lift{b}$ and, indeed, it resides 
in $\HH{}^1_{0,\dirichlet}(\domain) $. 

For the assumed linear-elastic behavior
according to~\eqref{eqn:strain} and~\eqref{eqn:stressstrain}, it follows straightforwardly that
\begin{subequations}
  \label{eqn:aGcontinuouscoercive}
  \begin{align}
    |a_{\Gamma}(u,v)|&\leq{}\bcA\,|{u}|_{1,\domain\setminus\Gamma}|v|_{1,\domain\setminus\Gamma}
    \label{eqn:aGcontinuous}
    \\
    |a_{\Gamma}(u,u)|&\geq{}\ucA\,|{u}|_{1,\domain\setminus\Gamma}^2
    \label{eqn:strongpositive}
  \end{align}
\end{subequations}
for all $u,v\in{}\HH^1(\domain\setminus\Gamma)$, where $\bcA$ and 
$\ucA$ denote the continuity and strong-positivity constants
of the elasticity tensor, respectively, and \mbox{$|\cdot|_{1,\domain\setminus\Gamma}$} represents the usual $\HH^1$-seminorm.
By virtue of Poincar\'e's inequality (see, e.g., \cite[Theorem 5.3.5]{BrennerScott2002}) there exists a bounded positive constant $C_{P}$ such that
\begin{equation}
  \label{eqn:Poincare}
  \|u\|_{1,\Omega\setminus\Gamma}\leq{}C_P|u|_{1,\Omega\setminus\Gamma}\qquad\forall{}u\in\HH^1_{0,\dirichlet}(\Omega\setminus\Gamma)
\end{equation}
Hence, the $\HH^1$-norm and $\HH^1$-semi norm are equivalent on $\HH^1_{0,\dirichlet}(\domain)$. Equations~\EQ{aGcontinuouscoercive}
and~\EQ{Poincare} imply that the bilinear forms~$a_{\Gamma}(\cdot,\cdot)$ and, accordingly, $a(\cdot,\cdot)$, are continuous and coercive. 
Moreover, it is easily verified that the linear 
form $l(\cdot)-\ab(\lift{b},\cdot):\HH^1(\domain)\to\reals$ in the right member of~\eqref{eqn:weak_varform1}
is continuous. Problem~\eqref{eqn:weak_varform1} therefore complies with the
conditions of the Lax-Milgram lemma (see, for instance, \cite[Theorem 2.7.7]{BrennerScott2002}) and, hence, it is well posed.

It is interesting to note that \eqref{eqn:weak_varform1} allows for a physical
interpretation that is very close to Volterra's classical construction for
dislocations, popularly known as the `Volterra knife': to make a cut in the
material, displace the two sides and hold them while welding the seam, and
finally release the sides so the material assumes its state of self-stressed
equilibrium. The initial cut and displacement is represented by the lift $
\lift{} $, which is not in equilibrium and is hence maintained by an external
load. The addition of $ \bar u $ represents the transition to a state of
equilibrium, by removing the external load but leaving the displacement intact.

\subsection{Finite element approximation}
\label{sec:fem}
Galerkin finite-element approximation methods for Volterra's dislocation problem~\eqref{eqn:weak_varform} are generally based on
a restriction of the weak formulation to a suitable finite dimensional subspace. The general
structure of finite-element methods can be found in many textbooks, for instance, \cite{Hughes:2000nx,Zienkiewicz:2000fl,ErnGuermond2004,Ciarlet:1991oq}.
We present here the main concepts and definitions for the ensuing exposition.

The approximation spaces in finite-element methods are generally subordinate to a mesh~$\mathcal{T}_h$, viz., a cover of the domain by non-overlapping
element domains~$\kappa\subset\Omega$. The subscript $h>0$ indicates the dependence of the mesh on a resolution 
parameter, for instance, the diameter of the largest element in the mesh. In general, we impose some auxiliary conditions
on the mesh, such as shape-regularity of the elements and conditions on the connectivity between elements; see, for instance, \cite{ErnGuermond2004,Di-Pietro:2012kx}
for further details. A finite-element approximation space $\VV_h^p\subset\HH^1_{0,\dirichlet}(\Omega)$
subordinate to~$\mathcal{T}_h$ can then be defined, for instance, as the subspace of vector-valued continuous element-wise polynomials of 
degree~$\leq{}p$ which vanish on~$\dirichlet$:
\begin{equation}
\label{eqn:PolynomialSpace}
\VV_h^p=\{v_h\in{}C^0(\overline{\Omega},\IR^N):(v_h)_i|_{\kappa}\in\mathbb{P}^p\text{ for all }\kappa\in\mathcal{T}_h,i=1,2,\ldots,N, v_h|_{\dirichlet}=0\}
\end{equation} 
with $\mathbb{P}^p$ the $N$-variate polynomials of degree $p$. Below, our interest is generally restricted to the $h$-dependence of the
approximation space and, accordingly, we will suppress~$p$. The finite-element approximation of~\eqref{eqn:weak_varform} based on an
approximation space~$\VV_h$ writes: find $u_h:=\bar{u}_h+\lift{b}$ with $\bar{u}_h\in \VV_h $ subject to
\begin{equation}
  \label{eqn:weak_varform1_fem}
  \qquad{}
  a(\bar{u}_h,v_h) = l(v_h)-\ab(\lift{b},v_h)
  \qquad
  \forall{}v_h\in{}\VV_h\,.
\end{equation}
We refer the right-most term in the right member of~\eqref{eqn:weak_varform1_fem} as the lift term.

\addtocounter{footnote}{2}
\renewcommand{\thefootnote}{\fnsymbol{footnote}}
Approximation properties of the Finite Element Method are generally investigated
on the basis of a sequence of meshes $\mathcal{T}_{\mathcal{H}}:=(\mathcal{T}_h)_{h\in\mathcal{H}}$, 
parametrized by a decreasing sequence of mesh parameters $\mathcal{H} = \{ h_1, h_2, \dots \} $ with $0$ as only accumulation point.
A sequence of meshes is called quasi uniform if there exist positive constants $\underline C$ and~$\overline{C}$, independent of~$h$, such that
$\underline{C}h \leq \mathrm{diam}(\kappa) \leq \overline C h $ for all $\kappa\in\mathcal{T}_h$ and all $h\in\mathcal{H}$.
Standard interpolation theory in Sobolev spaces (see, for instance, \cite{ErnGuermond2004,BrennerScott2002}) conveys that a sequence 
of approximation spaces $\VV_{\mathcal{H}}$ of the form~\eqref{eqn:PolynomialSpace} based on 
quasi-uniform meshes possesses the following approximation property: there exists a positive constant~$\const$\footnote{We use $\const$ to denote a generic positive constant, of which the value and connotation may change from one instance to the next, even within a single chain of expressions.} independent of~$h$ such that for all $h\in\mathcal{H}$,
all $k\geq{}0$ and both $m\in\{0,1\}$, it holds that
\begin{equation}
\label{eqn:OptimalInterpolation}
\inf_{v_h\in\VV_h}\|v-v_h\|_{m,\Omega}\leq\const{}h^{l+1-m}|v|_{l+1,\Omega}
\qquad\forall{}v\in\HH^{k+2}(\Omega)\cap\HH^1_{0,\dirichlet}(\Omega)
\end{equation}
with $l=\min\{p,k+1\}$. The estimate~\eqref{eqn:OptimalInterpolation} imparts that for all sufficiently smooth~$v\in\HH^1_{0,\dirichlet}(\Omega)$, 
the $\|\cdot\|_{m,\Omega}$-norm of the best approximation in~$\VV_h$ in that norm decays as $O(h^{p+1-m})$ as $h\to{}0$.

The lift $ \ell_b $ in~\eqref{eqn:weak_varform1_fem} is in principle arbitrary. However, 
the use of an arbitrary lift carries severe algorithmic complexity, as one has
to explicitly construct the lift and evaluate integrals involving products of
(gradients of) the lift and finite-element shape functions. Moreover, the evaluation
of these integrals by a suitable numerical integration scheme generally leads to a 
high computational complexity, because the fault is allowed to intersect
elements, and there are no efficient quadrature schemes to integrate the
discontinuous function that arises. Therefore, in practice, it is
convenient to integrate the lift in the finite-element setting. Provided that that
the fault coincides with element edges, a lift ${\lift{b}}_h$ is then constructed
in the broken approximation space:
\begin{equation}
\label{eqn:BrokenSpace}
\hat{\VV}_h
=
\big\{u\in{}C^0(\overline{\Omega}\setminus\Gamma):u|_{\Omega^\pm}\in(\VV_h)|_{\Omega^\pm}\big\}
\end{equation}
It is to be noted that the slip does not generally reside in $\jump{\hat{\VV}_h}$ and, accordingly,
$b$ is to be replaced by a suitable interpolant $b_h$. Moreover, if the fault does not coincide with
element edges, then it is to be replaced by an approximation subject to this condition. The aforementioned 
approach corresponds to the split-node method by Melosh~\cite{melosh81}, where the adjective `split' refers to 
the discontinuities between the elements in $ \Omega^+ $ and $ \Omega^- $ contiguous to~$\Gamma$. The split-node approach is
analogous to the standard treatment of Dirichlet boundary conditions; see, for
instance,~\cite[\S3.2.2]{ErnGuermond2004}. The split-node approach bypasses the aforementioned
complications of an arbitrary-lift approach
and the evaluation of the lift term comes essentially free of charge as part of the regular stiffness matrix. 
A fundamental disadvantage of the split-node approach, however, is that it requires that the fault
coincides with element edges, which connects the fault geometry to the geometry
and, generally, the topology of the mesh. As a result, computational primitives such
as the stiffness matrix and preconditioners for the stiffness matrix, which are contingent on
the mesh, cannot be reused for analyses of alternative fault geometries. This is a prohibitive restriction
if many dislocation geometries have to be considered, for instance, in inverse problems.

\section{The Weakly-enforced Slip Method}
\label{sec:wsm}

In section~\ref{sec:fem} we substantiated that the treatment of the lift term in the split-node approach, 
which is the natural counterpart of the standard treatment of Dirichlet boundary
conditions in finite-element approximations, is unsuitable if many fault geometries have to be
analysed, on account of the inherent dependence of the finite-element mesh on the fault geometry.

In this section we propose a new and fundamentally different treatment of the lift term that retains 
mesh independence:  the \emph{Weakly-enforced Slip Method} (WSM). Below, we present two different formal
derivations of the WSM formulation. The derivation in Section~\ref{sec:wsmderivation1} relies on a limit procedure.
In Section~\ref{sec:wsmderivation2}, we derive the WSM formulation on the basis of Nitsche's variational principle
for enforcing Dirichlet-type boundary conditions~\cite{Nitsche:1971fk}. 

\subsection{Collapsing the lift}
\label{sec:wsmderivation1}

Our aim is to derive a tractable finite-element approximation of Volterra's dislocation 
problem~\eqref{eqn:weak_varform}, in which the finite-element space and the bilinear form
and, accordingly, the stiffness matrix are independent of the fault geometry. 

In principle, the lift-based Galerkin formulation~\eqref{eqn:weak_varform1_fem} already
exhibits the appropriate form. However, as elaborated in Section~\ref{sec:fem}, the
corresponding finite-element formulation is intractable for general lift operators.
One can infer, however, that the complications engendered by a general lift operator
can be avoided by collapsing the support of the lift on the fault. The integration
of a discontinuous function in~$\Omega$ then reduces to the integration of a smooth function
on the dislocation. Numerical evaluation of integrals on the dislocation is feasible
given a parametrization of the fault. Moreover, the intricate explicit construction
of a lift is obviated, and only the slip distribution itself is required, which is 
presented as part of the problem specification. A further advantage of
collapsing the lift is that of localization: instead of having to
evaluate the lift term for
all shape functions of which the support intersects with the support of the
lift, only the shape functions of which the support intersects with the
dislocation have to be considered.

To derive the lift term corresponding to a collapsed lift, we consider a symmetric lift
$\ell_{b}^{\varepsilon}$ as illustrated in Fig.~\ref{fig:lift}, with compact support in an $\varepsilon $-neighborhood of the fault:
\begin{equation}
  \label{eqn:epsneighborhood}
  \Gamma^{\varepsilon}:=\{x\in\Omega:\mathrm{dist}(x,\Gamma)<\varepsilon\}.
\end{equation}
By virtue of the compact support of $\ell_b^{\varepsilon}$ in $ \Gamma^\varepsilon $, the lift 
term of \eqref{eqn:weak_varform1_fem} evaluates to
\begin{equation}
  a_\Gamma(\ell_{b}^{\varepsilon},v)
  = \int_{ \Gamma^\varepsilon \setminus\Gamma}
  \sigma(v) : \nabla \ell_{b}^{\varepsilon}
 = \int_\Gamma b \cdot \mean{ \sigma_\nu( v ) }
  -\int_{ \Gamma^\varepsilon\setminus\Gamma}
 \ell_{b}^{\varepsilon}\cdot\divstress( v )\,.
  \label{eqn:almostwsm}
\end{equation}
The first identity follows from the symmetry of the bilinear form in~\EQ{aba}. 
The second identity results from integration-by-parts and $ \jump{\ell_{b}^{\varepsilon}}=b$ 
and $\mean{\ell_{b}^{\varepsilon}}=0$. Let us note that the second identity is formal
in the sense that it requires more regularity of~$v$ than is actually provided by $\HH^1(\Omega)$.
We shall momentarily ignore this aspect, but it manifests itself in the analysis of the approximation
properties of the WSM formulation in Section~\ref{sec:convergence}. Without loss of generality,
we can assume $\ell_b^{\varepsilon}$ to be bounded independent of $\varepsilon$. The second
term in the final expression in~\eqref{eqn:almostwsm} then vanishes if $\Gamma^{\varepsilon}$ collapses
on~$\Gamma$. Therefore, formally passing to the limit in~\eqref{eqn:almostwsm}, we obtain
\begin{equation}
  \label{eqn:liftlimit}
a_\Gamma(\ell_{b}^{\varepsilon},v )
\xrightarrow{\varepsilon \to +0}
   \int_\manifold b \cdot \mean{ \sigma_\nu( v ) }
\end{equation}
According to~\eqref{eqn:liftlimit}, the lift term reduces to an integral on $ \Gamma $ in
the limit of collapsing the support of the lift onto the fault. The WSM formulation
corresponds to replacing the lift term $a_{\Gamma}(\ell_b,\cdot)$ in the right member
of~\eqref{eqn:weak_varform1_fem} by the limit functional according to~\eqref{eqn:liftlimit}:
\begin{framed} 
\noindent \textbf{Weakly-enforced Slip Method:} given a slip distribution
$ b : \Gamma \rightarrow \mathbb R^N $, find $ u_h \in \VV_h $ such that
\begin{equation}
  a( u_h, v_h ) = l( v_h ) - \int_\manifold \slip
  \cdot \mean{ \sigma_\nu( v_h ) }
  \qquad
  \forall{}v_h\in{}\VV_h\,.
\label{eqn:wsm} 
\end{equation} 
\end{framed}

\begin{figure}
  \centering
  \includegraphics{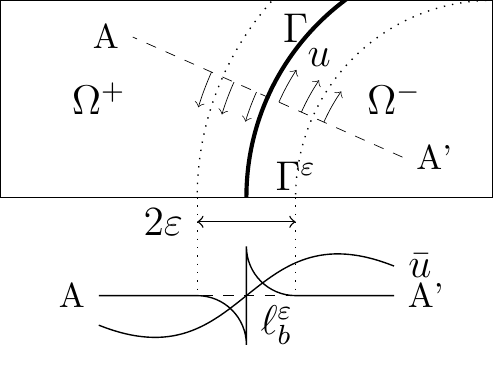}
  \caption{Schematic representation of an $\varepsilon$-local lift $
    \supp(\ell_{b}^{\varepsilon}) \in \Gamma^\varepsilon $, which is added to a continuous $
    \bar u $ to form the discontinuous displacement field $ u $.}
  \label{fig:lift}
\end{figure}
The nomenclature {\em Weakly-enforced Slip Method} serves to indicate that in~\eqref{eqn:wsm}
the slip discontinuity is weakly enforced in the right-hand side only, and does not appear
in the approximation space.

It is to be noted that although the WSM formulation~\eqref{eqn:wsm} is 
derived from the lift-based formulation~\eqref{eqn:weak_varform1_fem} by
collapsing the lift, in contrast to~\eqref{eqn:weak_varform1_fem} we do not
add a lift to the continuous complement $u_h$. Because $\VV_h\subset{}C^0(\Omega,\IR^N)$,
WSM thus yields a continuous approximation to the discontinuous solution of the
Volterra dislocation problem~\eqref{eqn:weak_varform}. This implies
that the approximation near the dislocation will inevitably be inaccurate. 
We will however show in Section~\ref{sec:convergence} that away from the dislocation,
the error in the WSM approximation converges optimally under mesh refinement.

\subsection{Alternative derivation via Nitsche's variational principle}
\label{sec:wsmderivation2}

To further elucidate the WSM formulation, we present in this section an alternative derivation of~\eqref{eqn:wsm}
based on {\em Nitsche's Variationsprinzip}~\cite{Nitsche:1971fk}. Nitsche presented in~\cite{Nitsche:1971fk} a
variational principle for weakly imposing Dirichlet boundary conditions in finite-element approximations of 
elliptic problems, i.e., without incorporating such essential boundary conditions in the approximation space.
Nitsche's variational principle can be extended to the Volterra dislocation problem~\eqref{eqn:weak_varform} 
to weakly impose the slip discontinuity. To specify this extension, we consider a suitable broken space 
$\hat{\VV}(h)$ which encapsulates the broken approximation space $\hat{\VV}_h$ and contains the solution~$u$ 
to the Volterra dislocation problem~\eqref{eqn:weak_varform}.
We define the quadratic functional $J:\hat{\VV}(h)\to\reals$:
\begin{equation}
\label{eqn:J(w)}
J(w)=\frac{1}{2}a_{\Gamma}(w,w)-\int_{\Gamma}\jump{w}\cdot\mean{\sigma_{\nu}(w)}+\frac{\psi}{2}\int_{\Gamma}\jump{w}^2\,,
\end{equation} 
for some suitable constant $\psi>0$, generally dependent on~$h$. Let $\check{\VV}_h$ denote either the broken approximation space~$\hat{\VV}_h$ or
the continuous approximation space~$\VV_h$ and consider the approximation $\check{u}_h\approx{}u$ according to:
\begin{equation}
\label{eqn:arginf}
\check{u}_h:=\arginf_{v_h\in\check{\VV}_h}J(u-v_h)
\end{equation}
Equation~\eqref{eqn:arginf} implies that $\check{u}_h$ satisfies the Kuhn-Tucker optimality condition $J'(u-\check{u}_h)(v_h)=0$ for all
$v_h\in\check{\VV}_h$, where $v\mapsto{}J'(w)(v)$ denotes the Fr\'echet derivative of $J$ at~$w$. For $J$ according to~\eqref{eqn:J(w)},
the optimality condition implies that $\check{u}_h\in\check{\VV}_h$ satisfies:
\begin{equation}
\label{eqn:SIPG}
\begin{aligned}
&
a_{\Gamma}(\check{u}_h,v_h)-\int_{\Gamma}\jump{\check{u}_h}\cdot\mean{\sigma_{\nu}(v_h)}-\int_{\Gamma}\jump{v_h}\cdot\mean{\sigma_{\nu}(\check{u}_h)}
+
\psi\int_{\Gamma}\jump{\check{u}_h}\cdot\jump{v_h}
\\
&\quad=
a_{\Gamma}(u,v_h)-\int_{\Gamma}\jump{u}\cdot\mean{\sigma_{\nu}(v_h)}-\int_{\Gamma}\jump{v_h}\cdot\mean{\sigma_{\nu}(u)}
+
\psi\int_{\Gamma}\jump{u}\cdot\jump{v}_h
\\
&\quad=
l(v_h)
-\int_{\Gamma}b\cdot\mean{\sigma_{\nu}(v_h)}
+
\psi\int_{\Gamma}b\cdot\jump{v}_h
\qquad\forall{}v_h\in\check{\VV}_h.
\end{aligned}
\end{equation}
The final identity follows by invoking integration-by-parts on $a_{\Gamma}(u,v_h)$, a rearrangement of terms
and the strong formulation of Volterra's dislocation problem in~\eqref{eqn:BVP}.

If the broken approximation space~$\hat{\VV}_h$ is inserted for~$\check{\VV}_h$, the optimality condition~\eqref{eqn:SIPG} can be 
reinterpreted as a symmetric-interior-penalty (SIP) discontinuous-Galerkin-type formulation; see, for instance,~\cite[Sec. 4.2]{Di-Pietro:2012kx}. In contrast to
standard discontinuous Galerkin formulations, the slip terms in the right-hand side, i.e., the terms containing $b$ in the ultimate expression in~\eqref{eqn:SIPG},
enforce the jump discontinuity at the fault. Convergence results for this formulation can be established in a similar manner as in~\cite{Nitsche:1971fk}.
For suitable stabilization parameters~$\psi$, the functional $J$ in~\eqref{eqn:J(w)} is equivalent to $\|\cdot\|_{1,\Omega\setminus\Gamma}$ and~\eqref{eqn:arginf}
implies quasi-optimal convergence of~$\check{u}_h$.

If the continuous approximation space~$\VV_h$ is inserted for~$\check{\VV}_h$, the terms containing~$\jump{\check{u}_h}$ and $\jump{v_h}$ vanish,
and we obtain the WSM formulation~\eqref{eqn:wsm}. Hence, WSM can indeed be interpreted as an extension of Nitsche's variational principle to
the Volterra dislocation problem with continuous approximation spaces. Furthermore, in view of $\hat{\VV}_h\supset\VV_h$, WSM can also be regarded
as a SIP discontinuous Galerkin formulation, based on a continuous subspace. One can infer that the WSM approximation retains the quasi-optimal
approximation property in $\|\cdot\|_{1,\Omega\setminus\Gamma}$. However, since the continuous approximation spaces applied in WSM are not dense 
in $H^1_{0,\dirichlet}(\Omega\setminus\Gamma)$, the immediate significance of this quasi-optimality for the approximation properties of
WSM is limited.

\section{Approximation properties of WSM}
\label{sec:convergence}

An analysis of the approximation properties of the Weakly-enforced Slip
Method is non-trivial, owing to the fact that in WSM one considers
approximations in $\HH^1_{0,\dirichlet}(\domain)$-conforming subspaces,
while the solution itself resides in
$\HH^1_{0,\dirichlet}(\domain\setminus\manifold)$, and the embedding of
$\HH^1_{0,\dirichlet}(\domain)$ in
$\HH^1_{0,\dirichlet}(\domain\setminus\manifold)$ is non-dense.
Essentially, we attempt to approximate a discontinuous function by a
continuous one and, in doing so, we incur an error that does not vanish
under mesh refinement. Standard techniques to assess global
approximation properties, on all of $\domain\setminus\manifold$, viz.,
C\'ea's lemma or the Strang lemmas~\cite[Lems.
2.25-27]{ErnGuermond2004}, therefore provide only partial information;
see Section~\ref{sec:globconv}. 

To provide a foundation for analyzing the approximation provided by WSM, we we
first recall some aspects of traces and tractions in
Section~\ref{sec:TracTrac}. Section~\ref{sec:globconv} investigates the global
approximation properties of WSM, i.e., on the entire domain.
Section~\ref{sec:locconv} establishes the local approximation properties of
WSM, i.e, on the domain excluding a neighborhood of the fault.

\subsection{Traces and tractions}
\label{sec:TracTrac}
To enable an analysis of the approximation behavior of WSM, some
elementary aspects of trace theory are required. For a comprehensive
overview, we refer to~\cite{Sayas:2009fk}. To make the theory applicable
to the Volterra dislocation problem, we must impose auxiliary smoothness
conditions on the elasticity tensor. In particular, we assume:
\begin{equation}
\label{eqn:smoothelast}
A_{ijkl}\in{}C^1\big(\overline{\Omega}\big)
\end{equation}
It is to be noted that~\EQ{smoothelast} implies that the elasticity tensor is $C^1$ continuous on the domain, 
including the boundary, and across the fault, including the dislocation. 
The $C^1$ continuity on the subdomains $\overline{\Omega^+}$ and~$\overline{\Omega^-}$ ensures that tractions
are well defined. 
The $C^1$ continuity across the fault is required to establish global convergence of the WSM approximation
in the $\LL^2$-norm and local convergence in the $\HH^1$-norm; see Sections \ref{sec:globconv} and~\ref{sec:locconv} below. 
Let us note that in tectonophysics the elasticity tensor is generally assumed to be uniformly constant in the domain.

Let $\omega\subset\reals^N$ denote an arbitrary connected domain with Lipschitz boundary. In particular,
recalling the partition of $\domain$ into the complementary subsets $\domain^{\pm}$, we envisage
$\omega\in\{\Omega{}^+,\Omega{}^-\}$.
We denote by $\widehat{\gamma}$ the restriction of a function in $\vctr{C}^1\smash[t]{(\overline{\omega})}$ 
to the boundary $\partial\omega$. By virtue of the density of $\vctr{C}^1(\omega)$ in $\HH^1(\omega)$, the 
operator can be extended to a linear continuous \emph{trace operator}, denoted by $\gamma$, 
from~$\HH^1(\omega)$ into~$\LL^2(\partial\omega)$. The image of $\gamma$ is denoted by $\smash[t]{\HH^{1/2}(\partial\omega)}$.
Considering a subset $\varkappa\subset\partial\omega$, we denote by $\gamma_{\varkappa}:=(\gamma(\cdot))|_{\varkappa}$ the 
composition of the trace operator and the restriction to $\varkappa$. The image of $\gamma_{\varkappa}$ restricted to
the class of functions $\smash[t]{\HH^1_{0,\partial\omega\setminus\varkappa}(\omega)}$ that vanish on $\partial\omega\setminus\varkappa$
is indicated by $\smash[t]{\HH^{1/2}_0(\varkappa)}$:
\begin{equation}
\HH^{1/2}_{0}(\varkappa)=\big\{\gamma_{\varkappa}u:u\in{}\HH^1_{0,\partial\omega\setminus\varkappa}(\omega)\big\}.
\end{equation}
The space $\smash[t]{\HH^{1/2}_0(\varkappa)}$ can be endowed with the norm:
\begin{equation}
\label{eqn:pmhalfnorm}
\|\lambda\|_{1/2,\varkappa}:=\inf{}\big\{\|u\|_{1,\omega}:u\in{}\HH^1_{0,\partial\omega\setminus\varkappa}(\omega), 
\gamma_{\varkappa}u=\lambda\big\}.
\end{equation}
with the obvious extension to $\HH^{1/2}(\partial\omega)$. There exist continuous right inverses
\begin{equation} 
\gamma^{-1}:\HH^{1/2}(\partial\omega)\to\HH^1(\omega),
\qquad
\gamma_{\varkappa}^{-1}:\HH^{1/2}_0(\varkappa)\to\HH^1_{0,\partial\omega\setminus\varkappa}(\omega),
\end{equation}
of $\gamma$ and $\gamma_{\varkappa}$. Such a right inverse is called a \emph{lifting} (or \emph{lift}) of the trace. It is to
be noted that lift operators are generally non-unique.

We denote by $\widehat{\sigma}_n:u\mapsto{}n\cdot\gamma(\sigma(u))$ the traction of a function 
in $\vctr{C}^2(\overline{\omega})$ 
on~$\partial\omega$, where $n$ denotes the exterior unit normal vector on $\partial\omega$. We define
\begin{equation}
\HH^1_{\subdivstress}(\omega):=\big\{u\in{}\HH^1(\omega):\divstress{}(u)\in{}\LL^2(\omega)\big\}.
\end{equation}
Applying index notation for transparency, the chain rule yields
\begin{equation}
\partial_j\sigma_{ij}(u)=(\partial_j\CC_{ijkl})\epsilon_{kl}(u)+\CC_{ijkl}(\partial_j\epsilon_{kl}(u))
\end{equation}
Therefore, the condition $\divstress(u)\in\LL^2(\omega)$ provides a meaningful condition on~$u$ 
if~$\partial_j\CC_{ijkl}\in{}L^{\infty}(\omega)$. For $\omega\in\{\Omega{}^+,\Omega{}^-\}$, this auxiliary 
condition on the elasticity tensor is satisfied under the standing assumption~\EQ{smoothelast}.
The vector space~$\HH^1_{\subdivstress}(\omega)$ is a Hilbert space under the inner-product associated with the
norm $\smash[b]{(\|\cdot\|_{1,\omega}^2+\|\divstress{(\cdot)}\|_{\omega}^2)^{1/2}}$.
The traction~$\widehat{\sigma}_n$ can be extended to a bounded linear operator, denoted by~$\sigma_n$, 
from $\HH^1_{\subdivstress}(\omega)$ into $\HH^{-1/2}(\partial\omega):=\smash[t]{\big[\HH^{1/2}(\partial\omega)\big]'}$, the dual space 
of $\HH^{1/2}(\partial\omega)$. For each $u\in\smash[t]{\HH^1_{\subdivstress}(\omega)}$, the functional $\sigma_n(u)$ acts on 
functions in $\smash[t]{\HH^{1/2}(\partial\omega)}$ by means of the following duality pairing:
\begin{equation}
\label{eqn:dualpair}
\langle{}\sigma_n(u),\lambda\rangle
=
\int_{\omega}\divstress{(u)}\cdot\gamma^{-1}(\lambda)
+
\int_{\omega}\sigma(u):\grad\gamma^{-1}(\lambda)
\end{equation}
One may note that for functions in $\vctr{C}^2(\overline{\omega})$, Equation~\EQ{dualpair} corresponds
to a standard integration-by-parts identity.
Continuity of the operator $\sigma_n$ thus defined follows from the sequence of bounds: 
\begin{equation}
\label{eqn:tracboundseq}
\begin{aligned}
\big|\langle{}\sigma_n(u),\lambda\rangle\big|
&\leq
\big\|\divstress{(u)}\big\|_{\omega}
\big\|\gamma^{-1}(\lambda)\big\|_{\omega}
+
\big\|\sigma(u)\big\|_{\omega}
\big|\gamma^{-1}(\lambda)\big|_{1,\omega}
\\
&\leq
\Big(\big\|\divstress{(u)}\big\|_{\omega}^2+\big\|\sigma(u)\big\|_{\omega}^2\Big)^{1/2}
\Big(\big\|\gamma^{-1}(\lambda)\big\|_{\omega}^2+\big|\gamma^{-1}(\lambda)\big|_{1,\omega}\Big)^{1/2}
\\
&\leq
\big(1+\bcA^2\big)^{1/2}
\big(\|u\|_{1,\omega}^2+\|\divstress{(u)}\|_{\omega}^2\big)^{1/2}
\big\|\gamma^{-1}(\lambda)\big\|_{1,\omega}
\end{aligned}
\end{equation}
and the continuity of the lifting of the trace from $\HH^{1/2}(\partial\omega)$ into $\HH^{1}(\omega)$. 
The dual space $\smash[t]{\big[\HH^{1/2}(\partial\omega)\big]'}$ is a Banach space under the norm
\begin{equation*}
\|v\|_{-1/2,\partial\omega}=
\sup_{\lambda\in\HH^{1/2}(\partial\omega)}\frac{\langle{}v,\lambda\rangle}{\|\lambda\|_{1/2,\partial\omega}}.
\end{equation*}
The restriction of the traction $(\widehat{\sigma}_n(\cdot))|_{\varkappa}$ to a subset $\varkappa\subset\partial\omega$ 
of the boundary can be extended to a bounded linear operator~$\sigma_{n,\varkappa}$ from $\HH^1_{\subdivstress}(\omega)$
into~$\HH^{-1/2}(\varkappa):=\smash[t]{\big[\HH^{1/2}_0(\varkappa)\big]'}$. The functional $\sigma_{n,\varkappa}(u)$
acts on functions in~$\smash[t]{\HH^{1/2}_0(\varkappa)}$ via the duality pairing:
\begin{equation}
\label{eqn:dualpair1}
\langle{}\sigma_{n,\varkappa}(u),\lambda\rangle
=
\int_{\omega}\divstress{(u)}\cdot\gamma_{\varkappa}^{-1}(\lambda)
+
\int_{\omega}\sigma(u):\grad\gamma_{\varkappa}^{-1}(\lambda)
\end{equation}
Continuity of the operator $\sigma_{n,\varkappa}$ thus defined follows in a similar manner as in~\EQ{tracboundseq}. 

For non-rupturing faults, it holds that $b\in{}\smash[t]{\HH^{1/2}_0(\Gamma)}$ and the above definitions apply 
without revisions. The slip 
discontinuity~\eqref{eqn:strong_dispjump} and the traction discontinuity~\eqref{eqn:strong_tracjump} are then
to be understood in the sense of traces and tractions outlined above. However, for rupturing faults, i.e., if the dislocation 
intersects with the boundary of the domain, then $b\notin\smash[t]{\HH^{1/2}_0(\Gamma)}$, and further consideration is required.
We can accommodate $b$ in the space:
\begin{equation}
\widetilde{\HH}{}^{1/2}(\Gamma):=\big\{\gamma_{\Gamma}u:u\in\HH^1(\omega)\big\}
\end{equation}
with $\omega\in\{\Omega^+,\Omega^-\}$. The space $\widetilde{\HH}{}^{1/2}(\Gamma)$ is a Banach space under the norm
\begin{equation}
\widetilde{\|\lambda\|}{}_{1/2,\Gamma}=
\inf\big\{\|u\|_{1,\omega}:u\in{}\HH^1(\omega),\gamma_{\Gamma}u=\lambda\big\}
\end{equation}
The principal complication pertaining to rupturing faults, is that the corresponding slip vectors
cannot be lifted into~$\HH^1_{0,\partial\omega\setminus\Gamma}(\omega)$, as traces of functions in $\HH^1(\omega)$ do 
not admit the discontinuity that would otherwise arise at the intersection of $\smash[t]{\overline{\Gamma}}$ 
and~$\smash[t]{\overline{\partial\omega\setminus\Gamma}}$.
Hence, we cannot use~\EQ{dualpair1} to define an extension of the restriction of the 
traction, $(\widehat{\sigma}{}_n(\cdot))|_{\Gamma}$, to a bounded 
linear operator from~$\HH^1_{\subdivstress}(\omega)$ into~$\smash[t]{\big[\widetilde{\HH}{}^{1/2}(\Gamma)\big]'}$. 
However, there exists a continuous right inverse 
$\smash[t]{\widetilde{\gamma}{}_{\Gamma}^{-1}:\widetilde{\HH}{}^{1/2}(\Gamma)\to\HH^1(\omega)}$ 
of the operator $\gamma_{\Gamma}$, for instance,
\begin{equation}
\label{eqn:rupturelift}
\widetilde{\gamma}{}_{\Gamma}^{-1}(\lambda)
=
\arginf\Big\{|u|_{1,\omega}:u\in{}\HH^1(\omega),\gamma_{\Gamma}u=\lambda\Big\}
\end{equation} 
Let us note that the image of the lift operator $\widetilde{\gamma}{}_{\Gamma}^{-1}$ corresponds to a harmonic function subject to
inhomogeneous Dirichlet conditions on $\Gamma$ with data $\lambda$ 
and a homogeneous Neumann condition on $\partial\omega\setminus\Gamma$. The lift operator $\widetilde{\gamma}{}_{\Gamma}^{-1}$ can 
be modified to include homogeneous Dirichlet conditions on $\dirichlet\subset\partial\omega\setminus\Gamma$ in the codomain, 
if necessary. The lift operator $\widetilde{\gamma}{}_{\Gamma}^{-1}$ enables us to extend $(\widehat{\sigma}{}_n)|_{\Gamma}$ to
a continuous linear operator:
\begin{equation}
\label{eqn:rupturesigman0}
{\widetilde{\sigma}}{}_{n,\Gamma}:
\Big\{u\in\HH^1_{\subdivstress}(\omega):\big\|\sigma_{n,\partial\omega\setminus\Gamma}(u)\big\|_{-1/2,\partial\omega\setminus\Gamma}=0\Big\}
\to
\big[\widetilde{\HH}{}^{1/2}(\Gamma)\big]'
\end{equation}
The functional $\widetilde{\sigma}{}_{n,\Gamma}(u)$ acts on functions in ${\widetilde{\HH}{}^{1/2}(\Gamma)}$ via the duality pairing:
\begin{equation}
\label{eqn:dualpair2}
\langle{}\widetilde{\sigma}_{n,\Gamma}(u),\lambda\rangle
=
\int_{\omega}\divstress{(u)}\cdot\widetilde{\gamma}{}_{\Gamma}^{-1}(\lambda)
+
\int_{\omega}\sigma(u):\grad \widetilde{\gamma}{}_{\Gamma}^{-1}(\lambda)
\end{equation}
Essentially, in~\EQ{rupturesigman0} and~\EQ{dualpair2}, we have defined the extension $\widetilde{\sigma}_{n,\Gamma}$ 
of the restriction of the traction to the dislocation,~$(\widehat{\sigma}_n(\cdot))|_{\Gamma}$, 
by restricting the domain of the extended operator to functions for which the
traction vanishes on $\partial\omega\setminus\Gamma$. This restriction in the definition is consistent with 
the standing assumption that an intersection of the dislocation with the boundary of the domain can only occur 
at Neumann boundaries.

In the analysis below, we restrict ourselves to non-rupturing faults. The analysis in Section~\ref{sec:globconv} however 
extends to rupturing faults by replacing the spaces and trace and traction operators for non-rupturing faults with those for 
rupturing faults.

\subsection{Global approximation properties of WSM}
\label{sec:globconv}
To assess the global approximation properties of WSM, we first construct an upper bound
on the functional $(b,\mean{\sigma_{\nu}(\cdot)})_{\Gamma}:\VV_h\to\reals$ in the right member of the WSM formulation.
It is to be noted that, in general, $\VV_h\not\subset{}\HH^1_{\divstress}(\Omega)$. Hence, the functional 
$(b,\mean{\sigma_{\nu}(\cdot)})_{\Gamma}:\VV_h\to\reals$ does not admit an interpretation as a duality pairing according 
to~\EQ{dualpair1}. Because $\VV_h$ is piecewise polynomial, however, an upper bound can be constructed on the basis 
of inverse and trace inequalities. We refer to~\cite{Di-Pietro:2012kx} for a comprehensive treatment of this
subject. Inverse and trace inequalities can generally be derived under suitable (sufficient) regularity conditions on the finite-element mesh; 
see~\cite[Chap.~1]{Di-Pietro:2012kx}. 
A detailed treatment of the conditions underlying inverse and trace inequalities is beyond the scope of this work. Instead, we shall
directly assume that a suitable discrete trace inequality holds. To formulate the assumption,
for a given sequence of partitions $\mathcal{T}_{\mathcal{H}}$ and for each $h\in\mathcal{H}$, we denote by $\mathcal{S}_h$ 
a dense cover of the fault by means of open intersections of the fault with element interiors or with element faces, i.e.,
\begin{equation}
\begin{aligned}
\mathcal{S}_h&=
\{s\subset\Gamma:s=\Gamma\cap\kappa\neq\emptyset\text{ for some }\kappa\in\mathcal{T}_h\}
\\
&\phantom{=}\cup
\{s\subset\Gamma:s=\mathrm{int}(\Gamma\cap\partial\kappa_0\cap\partial\kappa_1)\neq\emptyset\text{ for some }\kappa_0,\kappa_1\in\mathcal{T}_h,\kappa_0\neq\kappa_1\}
\end{aligned}
\end{equation}
See Figure~\FIG{segments} for an illustration. The separate treatment of element boundaries serves to ensure that subsets
of $\Gamma$ that coincide with element faces are separately included in~$\mathcal{S}_h$.
\begin{figure}
\centering
\includegraphics[scale=.4]{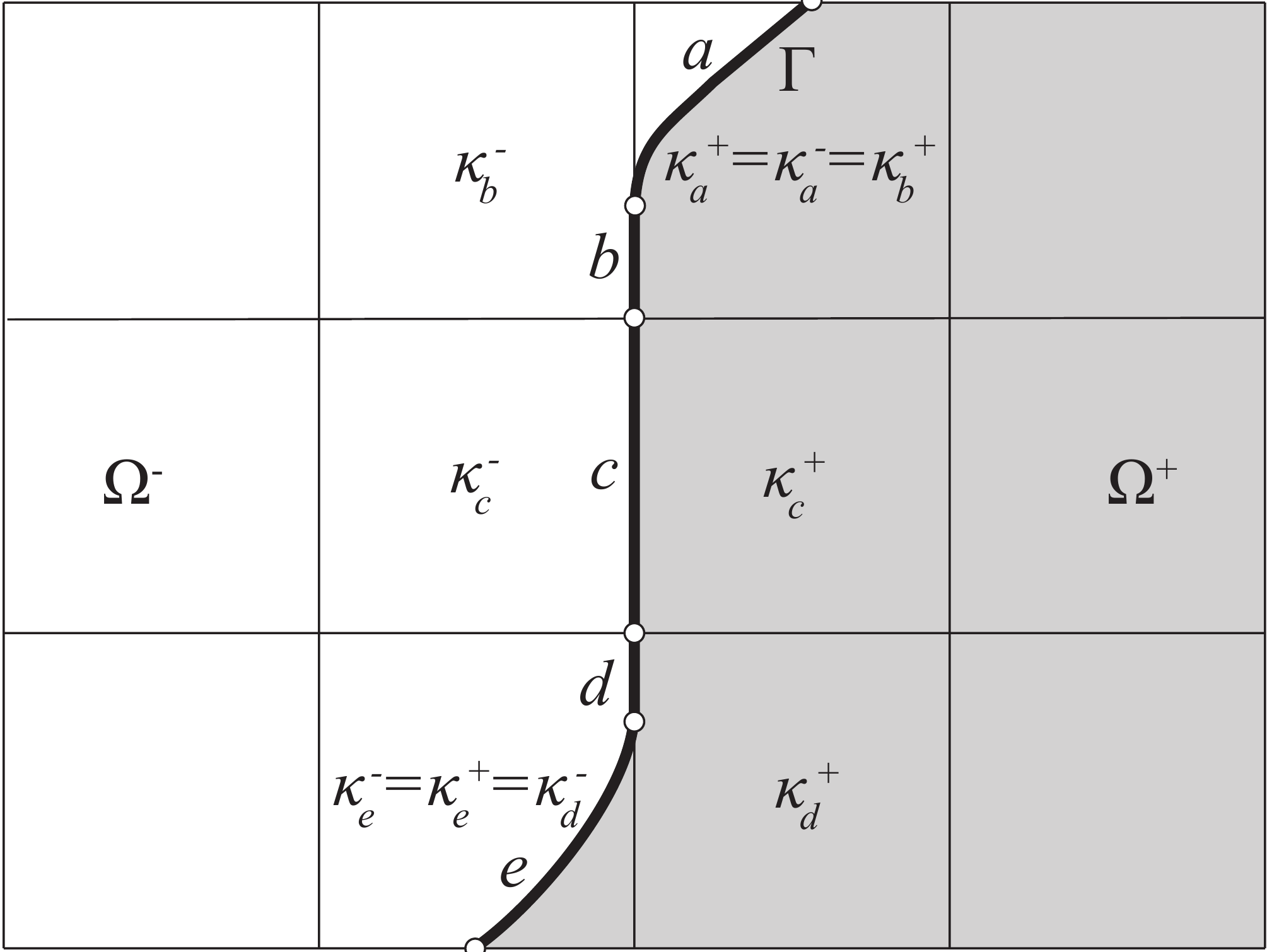}
\caption{Illustration of the covering of the fault~$\Gamma$ by segments $\mathcal{S}_h=\{a,b,c,b,e\}$,
and the corresponding elements $\{\kappa^+_{(\cdot)},\kappa^-_{(\cdot)}\}$. Because $b,c,d$ coincide
with element boundaries $\kappa^+_{(\cdot)}\neq\kappa^-_{(\cdot)}$ for these segments. The segments $a,e$ 
are interior to elements and, accordingly $\kappa^+_{(\cdot)}=\kappa^-_{(\cdot)}$ for these segments.%
\label{fig:segments}}
\end{figure}
To each segment~$s\in\mathcal{S}_h$, we associate a pair of contiguous elements $\{\kappa_s^+,\kappa_s^-\}$ such
that $s\subset\kappa^{\pm}\cup\partial\kappa^{\pm}$ and $\kappa^{\pm}\cap\Omega^{\pm}\neq\emptyset$.
If $s\subset\partial\kappa$ (resp. $s\subset\kappa$) for some $\kappa$ then $\kappa^+$ and $\kappa^-$
will be distinct (resp. identical).
We assume that the following discrete trace inequality holds 
for all $h\in\mathcal{H}$, all $s\in\mathcal{S}_h$ and all element-wise polynomial functions~$v_h$ 
of degree at most~$p$:
\begin{equation}
\label{eqn:DiscrTraceIneq}
\big(\diam(\kappa)\big)^{1/2}\|v_h\|_s\leq{}C_{\Gamma}\|v_h\|_{\kappa},
\end{equation}
for both $\kappa\in\{\kappa_s^+,\kappa_s^-\}$, for some $C_{\Gamma}>0$ independent of~$h$; cf.~\cite[Lemma~1.46]{Di-Pietro:2012kx}.
The constant $C_{\Gamma}$ is allowed to increase with the polynomial order~$p$.

\begin{lemma}[Continuity of the WSM linear form]
\label{thm:continuity}
Consider a manifold~$\Gamma\subset\Omega\subset\reals^N$, a slip distribution $\smash[t]{b\in\HH^{1/2}_0(\Gamma)}$ and 
a sequence of partitions $\mathcal{T}_{\mathcal{H}}$ such that for all~$h\in\mathcal{H}$, 
the discrete trace inequality~\EQ{DiscrTraceIneq} holds for all $s\in\mathcal{S}_h$ and all element-wise polynomial 
functions on~$\mathcal{T}_h$ and
\begin{equation}
\label{eqn:boundonb}
\|b\|_{\mathcal{T}_h,\Gamma}:=\bigg(\sum_{s\in\mathcal{S}_h}h_s^{-1}\|b\|_s^2\bigg)^{1/2}<\infty
\end{equation}
with $h_s$ the harmonic average of the diameters of the elements adjacent to~$s$, 
\begin{equation}
\label{eqn:hsdef}
\frac{1}{h_s}=\frac{1}{\diam(\kappa^+_s)}+\frac{1}{\diam(\kappa^-_s)}
\end{equation}
Then for all $h\in\mathcal{H}$, the linear form $(b,\mean{\sigma_{\nu}(\cdot)})_{\Gamma}:\VV_h\to\reals$ is continuous and
\begin{equation}
\label{eqn:lDbounded}
\big|(b,\mean{\sigma_{\nu}(v)})_{\Gamma}\big|
\leq{}
2^{-1/2}\,\bcA{}C_{\Gamma}M^{1/2}\,\|b\|_{\mathcal{T}_h,\Gamma}\,\|v\|_{1,\Omega}
\end{equation}
with $\bcA$ the continuity constant of the elasticity tensor in~\EQ{Abp}, $C_{\Gamma}$ the constant in the
discrete trace inequality~\EQ{DiscrTraceIneq} and~$M$ the maximum multiplicity of the multiset 
$\{\kappa\in\{\kappa^+_s,\kappa_s^-\}:s\in\mathcal{S}_h\}$.
\end{lemma}
\begin{remark}
The maximum multiplicity~$M$ in Lemma~\ref{thm:continuity} indicates the maximum number of occurrences of any one element in connection to any 
segment~$s\in\mathcal{S}_h$ as a member of the set~$\{\kappa^+_s,\kappa_s^-\}$. For instance, in Figure~\FIG{segments}, the 
element $\kappa_a^+=\kappa_a^-=\kappa_b^+$ has multiplicity $3$. One can infer that $M$ is bounded by the maximum number of faces
of any element in the mesh, increased by $2$ for interior segments.
\end{remark}
\begin{proof}
We first separate the integral on~$\Gamma$ into a sum of contributions
from the segments and apply~\EQ{Abp} and the Cauchy-Schwarz inequality to obtain
\begin{equation}
\label{eqn:intsplit}
\begin{aligned}
\big|(b,\mean{\sigma_{\nu}(v_h)})_{\manifold}\big|
&=
\bigg|
2^{-1}
\sum_{s\in{S}_h}
\big(b,\nu\cdot\gamma^+_{\Gamma}(\sigma(v_h))+\nu\cdot\gamma^-_{\Gamma}(\sigma(v_h))\big)_s\bigg|
\\
&\leq
2^{-1}\bcA
\sum_{s\in{S}_h}
\|b\|_s\Big(\big\|\gamma^+_{\Gamma}(\grad{}v_h)\big\|_s+\big\|\gamma^-_{\Gamma}(\grad{}v_h)\big\|_s\Big)
\end{aligned}
\end{equation}
where $\gamma_{\Gamma}^{\pm}(\cdot)$ denotes the trace of $(\cdot)$ from within $\Omega^{\pm}$.
Noting that~$v_h\in\VV_h$ is element-wise polynomial, we deduce from the discrete trace inequality~\EQ{DiscrTraceIneq}, the
arithmetic-geometric mean inequality and~\EQ{hsdef}: 
\begin{equation}
\label{eqn:intsplit2}
\begin{aligned}
\big|(b,\mean{\sigma_{\nu}(v_h)})_{\manifold}\big|
&\leq
2^{-1}\bcA{}C_{\Gamma}
\sum_{s\in{S}_h}
h_s^{-1/2}
\|b\|_s\big(\|\grad{}v_h\|_{\kappa_s^+}+\|\grad{}v_h\|_{\kappa_s^-}\big)
\\
&\leq
2^{-1/2}\,\bcA{}C_{\Gamma}
\bigg(
\sum_{s\in{S}_h}
h_s^{-1}
\|b\|_s^2\bigg)^{1/2}
\bigg(
\sum_{s\in{S}_h}
\|\grad{}v_h\|_{\kappa_s^+}^2
+
\|\grad{}v_h\|_{\kappa_s^-}^2\bigg)^{1/2}
\\
&\leq
2^{-1/2}\,\bcA{}C_{\Gamma}M^{1/2}
\,\|b\|_{\mathcal{T}_h,\Gamma}\,
\|v_h\|_{1,\Omega}
\end{aligned}
\end{equation}

\end{proof}

To determine the global approximation properties of the WSM formulation in the $\HH^1$-norm, 
we note that the WSM formulation is \emph{inconsistent}: The solution of the weak formulation of the 
Volterra dislocation problem~\EQ{weak_varform} violates the WSM weak form~\EQ{wsm} by 
$(b,\mean{\sigma_{\nu}(v)})_{\Gamma}$, for all $v\in\VV_h$. The second Strang lemma~\cite[Lemma 2.25]{ErnGuermond2004} then provides the
following characterization of the global approximation properties of WSM:
\begin{theorem}[Global approximation properties of WSM in the $\HH^1$-norm]
\label{thm:Strang2bound}
Assume that the conditions of Lemma~\ref{thm:continuity} hold. Let $u\in\HH^1_{0,\dirichlet}(\Omega\setminus\Gamma)$ 
denote the solution to the Volterra dislocation problem~\EQ{weak_varform} and let $u_h\in\VV_h$ denote its
WSM approximation according to~\EQ{wsm}. It holds that
\begin{align}
\label{eqn:Strang2bound}
\|u-u_h\|_{1,\Omega\setminus\Gamma}
&\leq
\big(1+C_P\ucA^{-1}\bcA\big)\inf_{v_h\in\VV_h}\|u-v_h\|_{1,\Omega\setminus\Gamma} \nonumber\\
&\quad +
C_P\ucA^{-1}\sup_{v_h\in\VV_h\setminus\{0\}}\frac{|(b,\mean{\sigma_{\nu}(v_h)})_{\Gamma}|}{\|v_h\|_{1,\Omega}}
\end{align}
\end{theorem}
\begin{proof}
We first recall that the bilinear form $a_{\Gamma}(\cdot,\cdot)$ is bounded on $\HH^1_{0,\dirichlet}(\Omega\setminus\Gamma)\times\VV_h$ and
coercive on $\VV_h\times\VV_h$; see~\EQ{aGcontinuouscoercive} and~\EQ{Poincare}. For arbitrary $v_h\in\VV_h$, we have
the chain of inequalities:
\begin{equation}
\begin{aligned}
\|u_h-v_h\|_{1,\Omega\setminus\Gamma}^2&\leq{}C_P\ucA^{-1}{}|a_{\Gamma}(u_h-v_h,u_h-v_h)|
\\
&=
C_P\ucA^{-1}|a_{\Gamma}(u_h-u,u_h-v_h)+a_{\Gamma}(u-v_h,u_h-v_h)|
\\
&\leq
C_P\ucA^{-1}\big(|(b,\mean{\sigma_{\nu}(u_h-v_h)})_{\Gamma}|
+
\bcA|u-v_h|_{1,\Omega\setminus\Gamma}|u_h-v_h|_{1,\Omega\setminus\Gamma}\big)
\end{aligned}
\end{equation}
which leads to
\begin{equation}
\begin{aligned}
\|u_h-v_h\|_{1,\Omega\setminus\Gamma}
\leq
C_P\ucA^{-1}\sup_{w_h\in\VV_h\setminus\{0\}}\frac{|(b,\mean{\sigma_{\nu}(w_h)})_{\Gamma}|}{\|w_h\|_{1,\Omega}}
+
C_P\ucA^{-1}\bcA\|u-v_h\|_{1,\Omega\setminus\Gamma}
\end{aligned}
\end{equation}
The bound~\EQ{Strang2bound} then follows from the triangle inequality.
\end{proof}

For quasi-uniform meshes,
Lemma~\ref{thm:continuity} and Theorem~\ref{thm:Strang2bound} lead to a simple asymptotic characterization of the global
approximation properties of WSM in the $\HH^1$-norm. This characterization is detailed in the following corollary:
\begin{corollary}
\label{thm:CorH1global}
Assume that the conditions of Lemma~\ref{thm:continuity} hold and that the sequence of partitions $\mathcal{T}_{\mathcal{H}}$ 
is quasi-uniform with respect
to the mesh parameter, i.e., for all $h\in\mathcal{H}$ there exist constants $\overline{C}>0$ and $\underline{C}>0$ independent of $h$
such that $\underline{C}h\leq\mathrm{diam}(\kappa)\leq\overline{C}h$ for all $\kappa\in\mathcal{T}_h$. It then holds that
\begin{equation}
\|u-u_h\|_{1,\Omega\setminus\Gamma}\leq\const{}h^{-1/2}
\end{equation}
as $h\to{}0$.
\end{corollary}
\begin{proof}
Subject to the quasi-uniformity condition on the sequence of partitions, we have
\begin{equation}
\label{eqn:Chb1}
\begin{aligned}
\|b\|_{\mathcal{T}_h,\Gamma}
=
\bigg(\sum_{s\in\mathcal{S}_h}h_s^{-1}\|b\|_s^2\bigg)^{1/2}
\leq
2\underline{C}^{-1/2}h^{-1/2}\bigg(\sum_{s\in\mathcal{S}_h}\|b\|_s^2\bigg)^{1/2}
=
2\underline{C}^{-1/2}h^{-1/2}\|b\|_{\Gamma}
\end{aligned}
\end{equation}
It then follows from~\EQ{intsplit2} and~\EQ{Strang2bound} that
\begin{equation}
\begin{aligned}
\|u-u_h\|_{1,\Omega\setminus\Gamma}
&\leq
\big(1+C_P\ucA^{-1}\bcA\big)\inf_{v_h\in\VV_h}\|u-v_h\|_{1,\Omega\setminus\Gamma} \\
&\quad +
2^{-1/2}
C_P\ucA^{-1}\bcA{}C_{\Gamma}M^{1/2}\|b\|_{\mathcal{T}_h,\Gamma}
\\
&\leq
\const+\const{}h^{-1/2}
\end{aligned}
\end{equation}
and the assertion follows in the limit $h\to{}0$.
\end{proof}

The potential divergence of the bound in Corollary~\ref{thm:CorH1global} does of course not immediately
imply that the error in the WSM approximation itself increases as $h\to{}0$. However, Theorem~\ref{thm:ANWSM} below asserts that $\|u-u_h\|_{\Omega}$ vanishes as
$h\to{}0$. From this result, the continuity of $u_h$ and the discontinuity of $u$, one can infer that, 
indeed, $\|u-u_h\|_{1,\Omega\setminus\Gamma}$ must diverge as $h\to{}0$. 

Theorem~\ref{thm:Strang2bound} conveys that for each $h>0$, it holds that $u-u_h\in\LL^2(\Omega)$. Based on the Aubin-Nitsche Lemma,
we can can therefore construct an estimate of the error in the WSM approximation in the $\LL^2$\nobreakdash-norm. The estimate is formulated
in Theorem~\ref{thm:ANWSM} below. Some auxiliary conditions on the domain~$\Omega$ are required, as specified in the premises of the theorem.
\begin{theorem}[Global approximation properties of WSM in the $\LL^2$-norm]
\label{thm:ANWSM}
Assume that the conditions of Lemma~\ref{thm:continuity} hold. In addition, assume that $\Omega$ is convex or of class $C^2$. 
Let $u$ denote the solution to~\EQ{weak_varform} and let $u_h\in\VV_h$ denote the WSM approximation according to~\EQ{wsm}. 
Let $\TT_h$ denote the intersection of $\gamma_{\Gamma}(\VV_h)$ with~$\HH^{1/2}_0(\Gamma)$. 
It holds that
\begin{equation}
\label{eqn:Thm53}
\begin{aligned}
\big|(u-u_h,\varphi)\big|
&\leq
\const\Big(|u-u_h|_{1,\Omega\setminus\Gamma}+\|b\|_{\mathcal{T}_h,\Gamma}\Big)\inf_{v_h\in\VV_h}|\zvp-v_h|_{1,\Omega}
\\
&\quad
+
\const
\inf_{\lambda_h\in\TT_h}
\Big(\|\vp\|_{\Omega}\big\|b-\lambda_h\big\|_{\Gamma}
+ \cdots \\
& \phantom{\quad +} \quad \big(\|\lambda_h\|_{1/2,\Gamma}
+
\|\lambda_h\|_{\mathcal{T}_h,\Gamma}\big)
\inf_{w_h\in\VV_h}
|\zvp-w_h|_{1,\Omega\setminus\Gamma}\Big)
\end{aligned}
\end{equation}
for arbitrary $\vp\in\LL^2(\Omega)$ and a corresponding $\zvp\in\HH^2(\Omega)\cap\HH^1_{0,\dirichlet}(\Omega)$ such that
$\|\zvp\|_{2,\Omega}\leq\const\|\vp\|_{\Omega}$, with~$\|\cdot\|_{\mathcal{T}_h,\Gamma}$ defined by~\EQ{boundonb}.
\end{theorem} 

\begin{proof}
For arbitrary $\varphi\in\LL^2(\Omega)$, consider the dual problem (in the sense of distributions):
\begin{subequations}
\label{eqn:dualproblem}
\begin{alignat}{2}
-\divstress(\zvp)&=\vp&\qquad&\text{in }\Omega
\label{eqn:dualproblemA}
\\
\zvp&=0&\qquad&\text{on }\dirichlet
\label{eqn:dualdirichlet}
\\
\sigma_n(\zvp)&=0&\qquad&\text{on }\neumann
\label{eqn:dualneumann}
\end{alignat}
\end{subequations}
or, equivalently, in weak form:
\begin{equation}
\label{eqn:dualproblemweak}
z\in\HH^1_{0,\dirichlet}(\Omega):
\qquad
a_{\Gamma}(v,\zvp)=(v,\vp)_{\Omega}
\qquad
\forall{}v\in\HH^1_{0,\dirichlet}(\Omega)
\end{equation}
By virtue of the smoothness conditions on the elasticity tensor~\EQ{smoothelast}, the conditions on the domain,
and the smoothness of the (homogeneous) boundary data in~\EQ{dualdirichlet} and~\EQ{dualneumann}, the dual 
problem~\EQ{dualproblem} possesses an elliptic-regularity property (see, e.g.,~\cite{Evans:2009ph,Brenner:2004fk}).
The regularity property implies that the dual solution resides in $\HH^2(\Omega)\cap\HH^1_{0,\dirichlet}(\Omega)$ and satisfies the 
estimate~$\|\zvp\|_{2,\Omega}\leq\const{}\|\vp\|_{\Omega}$. 

Denoting by $\lift{b}\in\HH^1_{0,\dirichlet}(\Omega\setminus\Gamma)$ a suitable lift of~$b$ such that~$\jump{\lift{b}}=b$, it holds that
$u-\lift{b}\in\HH^1_{0,\dirichlet}(\Omega)$. Moreover, the WSM approximation $u_h$ resides in~$\HH^1_{0,\dirichlet}(\Omega)$.
The dual problem~\EQ{dualproblemweak} therefore gives:
\begin{equation}
\label{eqn:AN1}
\begin{aligned}
(u-u_h,\varphi)_{\Omega}
&=
(u-\lift{b},\varphi)_{\Omega}+(\lift{b},\varphi)_{\Omega}-(u_h,\varphi)_{\Omega}
\\
&=
a_{\Gamma}(u-\lift{b},\zvp)+(\lift{b},\varphi)_{\Omega}-a_{\Gamma}(u_h,\zvp)
\\
&=
a_{\Gamma}(u,\zvp)-a_{\Gamma}(u_h,\zvp)-\big(a_{\Gamma}(\lift{b},\zvp)-(\lift{b},\varphi)_{\Omega}\big)
\end{aligned}
\end{equation}
By means of~\EQ{dualpair}, the term in parenthesis in the ultimate expression can be identified as the weak formulation of
$(b,\mean{\sigma_{\nu}(\zvp)})_{\Gamma}$. However, for~$\zvp\in{}\HH^2(\Omega)\cap\HH^1_{0,\dirichlet}(\Omega)$, the trace theorem asserts
that $\smash[t]{\sigma_{\nu}(\zvp)\in\LL^2(\Gamma)}$ and $(b,\mean{\sigma_{\nu}(\zvp)})_{\Gamma}$ coincides with its extension to a duality pairing.
From the weak form of the Volterra dislocation problem~\EQ{weak_varform} it moreover follows that $a_{\Gamma}(u,\zvp)=l(\zvp)$.
The Galerkin-orthogonality property of the WSM approximation~\EQ{wsm} then yields the following identities:
\begin{equation}
\label{eqn:AN2}
\begin{aligned}
(u-u_h,\vp)_{\Omega}
&=
l(\zvp-v_h)-a_{\Gamma}(u_h,\zvp-v_h)+(b,\mean{\sigma_{\nu}(v_h)})_{\Gamma}-(b,\mean{\sigma_{\nu}(\zvp)})_{\Gamma}
\\
&=
a_{\Gamma}(u-u_h,\zvp-v_h)-(b,\mean{\sigma_{\nu}(\zvp)}-\mean{\sigma_{\nu}(v_h)})_{\Gamma}
\end{aligned}
\end{equation}
which are valid for all $v_h\in\VV_h$. 

For the first term in the final expression in~\EQ{AN2} we can construct an appropriate bound without further digression.
However, the second term, pertaining to the difference between the average traction of $\smash[t]{z\in\HH^2(\Omega)\cap\HH^1_{0,\dirichlet}(\Omega)}$
and the average of the direct evaluation of the traction of the finite-element function $v_h\in\VV_h$, is more difficult to estimate.
The essential complication in constructing an estimate, is that the direct evaluation of the traction of the Galerkin finite-element 
approximation of~\EQ{dualproblemweak} in~$\VV_h$ does not coincide with the weak formulation, because $\VV_h\notin\HH^1_{\divstress}(\Omega)$;
see also~\cite{Brummelen:2012fk}. Approximation results are available for the weak formulation of the traction (see, for instance, \cite{Brezzi:2001fk})
but, to our knowledge, not for the direct formulation. 

To estimate the second term in~\EQ{AN2}, we will use an auxiliary Nitsche-type approximation~\cite{Nitsche:1971fk}
to the dual problem. We consider the broken approximation spaces~$\hat{\VV}_h$ according to~\eqref{eqn:BrokenSpace}.
Moreover, we define $\hat{\VV}(h):=\big(\HH^2(\Omega\setminus\Gamma)\cap\HH^1_{0,\dirichlet}(\Omega\setminus\Gamma)\big)+\hat{\VV}_h$.
We equip $\hat{\VV}(h)$ with the mesh-dependent inner product
\begin{equation}
(u,v)_{\hat{\VV}(h)}=a_{\Gamma}(u,v)+\zeta\sum_{s\in\mathcal{S}_h}h_s^{-1}(\jump{u},\jump{v})_s
\end{equation}
and the corresponding norm $\smash[b]{\|\cdot\|_{\hat{\VV}(h)}}$. We note that the embedding of $\HH^2(\Omega\setminus\Gamma)$ into~$\hat{\VV}(h)$ is continuous,
i.e., $\smash[t]{\|v\|_{\hat{\VV}(h)}\leq\const\|v\|_{2,\Omega\setminus\Gamma}}$ for all $\smash[t]{v\in\HH^2(\Omega\setminus\Gamma)}$. 
We define the bilinear operator $\hat{a}_{\Gamma}:\VV(h)\times\VV_h\to\IR$ according to:
\begin{equation}
\hat{a}_{\Gamma}(u,v)
=
a_{\Gamma}(u,v)-(\jump{v},\mean{\sigma_{\nu}(u)})_{\Gamma}+\theta\sum_{s\in\mathcal{S}_h}(\jump{u},\jump{v})_s
\end{equation}
with $\theta$ a suitable constant. For a suitable choice of the constants $\zeta$ and $\theta$, the bilinear form in the left member of~\EQ{dualnitsche}
is continuous and coercive on $\hat{\VV}_h\times\hat{\VV}_h$ and $\hat{a}_{\Gamma}(u,\cdot):\hat{\VV}_h\to\IR$ represents a continuous linear functional for all
$u\in\VV(h)$. Let ${\zvp}_h\in\hat{\VV}_h$ denote the solution to the following Nitsche-type projection problem:
\begin{equation}
\label{eqn:dualnitsche}
{\zvp}_h\in\hat{\VV}_h:\qquad
\hat{a}_{\Gamma}({\zvp}_h,w_h)
=
\hat{a}_{\Gamma}(\zvp,w_h)
\qquad\forall{}w_h\in\hat{\VV}_h
\end{equation}
By virtue of the continuity and coercivity of $\hat{a}_{\Gamma}$ on $\hat{\VV}_h\times\hat{\VV}_h$ and the continuity of $\hat{a}_{\Gamma}(\zvp,\cdot)$ on~$\hat{\VV}_h$,
the projection problem~\EQ{dualnitsche} defines a unique element ${\zvp}_h\in\hat{\VV}_h$ which satisfies
\begin{equation}
\label{eqn:CeaAN}
\|\zvp-{\zvp}_h\|_{\hat{\VV}(h)}
\leq
\const\inf_{w_h\in\hat{\VV}_h}\|\zvp-w_h\|_{\hat{\VV}(h)}
\leq
\const\inf_{w_h\in\VV_h}|\zvp-w_h|_{1,\Omega}
\end{equation}
The first inequality in~\EQ{CeaAN} is a straightforward consequence of C\'ea's lemma. The second inequality follows
from $\VV_h\subset\hat{\VV}_h$, the continuity of functions in~$\VV_h$ and~\EQ{aGcontinuous}. 

By adding a suitable partition of zero to the second term in~\EQ{AN2} and applying the triangle inequality, we obtain:
\begin{equation}
\label{eqn:AN3}
\big|(b,\mean{\sigma_{\nu}(\zvp)}-\mean{\sigma_{\nu}(v_h)})_{\Gamma}\big|
\leq
\big|(b,\mean{\sigma_{\nu}(\zvp-{\zvp}_h)})_{\Gamma}\big|+\big|(b,\mean{\sigma_{\nu}(v_h-{\zvp}_h)})_{\Gamma}\big|
\end{equation}
For the first term in the right-member of~\EQ{AN3}, we derive from~\EQ{dualnitsche} and~\EQ{CeaAN}:
\begin{equation}
\label{eqn:1stterm0}
\begin{aligned}
\big|(b,\mean{\sigma_{\nu}(\zvp-{\zvp}_h)})_{\Gamma}\big|
&=
\bigg|\big(b-\jump{w_h},\mean{\sigma_{\nu}(\zvp-{\zvp}_h)}\big)_{\Gamma}+a_{\Gamma}(w_h,\zvp-{\zvp}_h) \cdots \\
&\quad -\theta\sum_{s\in\mathcal{S}_h}h_s^{-1}(\jump{w_h},\jump{{\zvp}_h})_s\bigg|
\\
&\leq
\big\|b-\jump{w_h}\big\|_{\Gamma}\big(\|\mean{\sigma_{\nu}(\zvp)}\|_{\Gamma} + \|\mean{\sigma_{\nu}({\zvp}_h)}\|_{\Gamma}\big) \\
&\quad + \bcA|w_h|_{1,\Omega\setminus\Gamma}|\zvp-{\zvp}_h|_{1,\Omega\setminus\Gamma} \\
&\quad + \theta\big\|\jump{w_h}\big\|_{\mathcal{T}_h,\Gamma}
\big\|\jump{{\zvp}_h}\big\|_{\mathcal{T}_h,\Gamma}
\end{aligned}
\end{equation}
for all $w_h\in\hat{\VV}_h$. The trace theorem implies $\|\mean{\sigma_{\nu}(\zvp)}\|_{\Gamma}\leq\const\|\zvp\|_{2,\Omega}\leq\const\|\vp\|_{\Omega}$.
The continuity of $\hat{a}_{\Gamma}$ and~$a_{\Gamma}$ implies that $\|\mean{\sigma_{\nu}({\zvp}_h)}\|_{\Gamma}\leq\const\|{\zvp}_h\|_{\hat{\VV}(h)}$.
Moreover, by virtue of~\EQ{CeaAN} and the continuity of the embedding of $\HH^2(\Omega)$ into~$\smash[t]{\hat{\VV}(h)}$
it holds that $\|{\zvp}_h\|_{\hat{\VV}(h)}\leq\const\|\zvp\|_{\hat{\VV}(h)}\leq\const\|\zvp\|_{2,\Omega}\leq\const\|\vp\|_{\Omega}$. 
Hence, we have $\|\mean{\sigma_{\nu}(\zvp)}\|_{\Gamma}+\|\mean{\sigma_{\nu}({\zvp}_h)}\|_{\Gamma}\leq\const\|\vp\|_{\Omega}$.
Noting that $\smash[t]{\TT_h =\gamma_{\Gamma}(\VV_h)\cap\HH^{1/2}_0(\Gamma)}$ coincides with $\{\jump{w_h}:w_h\in\hat{\VV}_h\}$,
for all $\lambda_h\in{\TT}_h$ there exists a $w_h\in\hat{\VV}_h$ such that
$\jump{w_h}=\lambda_h$ and $|w_h|_{1,\Omega\setminus\Gamma}\leq\const\|\lambda\|_{1/2,\Gamma}$. From~\EQ{CeaAN}--\EQ{1stterm0} we deduce:
\begin{equation}
\label{eqn:1stterm}
\begin{aligned}
\big|(b,\mean{\sigma_{\nu}(\zvp-{\zvp}_h)})_{\Gamma}\big|
&\leq
\const
\Big(\|\vp\|_{\Omega}\big\|b-\lambda_h\big\|_{\Gamma}
+ \cdots \nonumber\\
&\quad \big(\|\lambda_h\|_{1/2,\Gamma}
+
\|\lambda_h\|_{\mathcal{T}_h,\Gamma}\big)
\inf_{w_h\in\VV_h}
|\zvp-w_h|_{1,\Omega\setminus\Gamma}\Big)
\end{aligned}
\end{equation}
for all $\lambda_h\in\TT_h$, with $\|\cdot\|_{\mathcal{T}_h,\Gamma}$ according to~\EQ{boundonb}.
For the second term in the right-member of~\EQ{AN3} we deduce from~\EQ{intsplit2}, the triangle inequality and~\EQ{CeaAN}:
\begin{equation}
\label{eqn:2ndpart}
\begin{aligned}
\big|(b,\mean{\sigma_{\nu}(v_h-{\zvp}_h)})_{\Gamma}\big|
&\leq
2^{-1/2}\bcA{}C_{\Gamma}M^{1/2}\|b\|_{\mathcal{T}_h,\Gamma}|v_h-{\zvp}_h|_{1,\Omega\setminus\Gamma}
\\
&\leq
\const \|b\|_{\mathcal{T}_h,\Gamma}\Big(|\zvp-v_h|_{1,\Omega}+\inf_{w_h\in\VV_h}|\zvp-w_h|_{1,\Omega}\Big)
\end{aligned}
\end{equation}
with $\const$ independent of~$h$.
Collecting the results in~\EQ{AN2}--\EQ{2ndpart} and taking the infimum with respect to $v_h$ and $\lambda_h$, one obtains
the estimate in~\EQ{Thm53}.
\end{proof}

Under slightly stronger conditions on the regularity of the slip distribution~$b$, we can derive from Theorem~\ref{thm:ANWSM} a 
straightforward characterization of the asymptotic approximation properties of WSM in the $\LL^2$-norm 
for quasi-uniform meshes in the limit as $h\to{}0$:
\begin{corollary}
\label{thm:CorL2global}
Assume that the conditions of Corollary~\ref{thm:CorH1global} and
Theorem~\ref{thm:ANWSM} hold and,
moreover,~$b\in\HH^1(\Gamma)\cap\HH^{1/2}_0(\Gamma)$. Let $u$ denote the
solution to~\EQ{weak_varform} and let $u_h\in\VV_h$ denote the WSM
approximation according to~\EQ{wsm}. It holds that
\begin{equation}
\label{eqn:CorL2global}
\|u-u_h\|_{\Omega}\leq\const{}h^{1/2}
\qquad
\end{equation}
as $h\to{}0$.
\end{corollary}

\begin{proof}
Based on the estimate $\|\zvp\|_{2,\Omega}\leq\const\|\vp\|_{\Omega}$, it follows from
standard interpolation theory in Sobolev spaces (see, for instance, \cite{ErnGuermond2004,BrennerScott2002}) that
$\inf_{v_h\in\VV_h}|z-v_h|_{1,\Omega}\leq{}\const{}h\|\vp\|_{\Omega}$.
Similarly, for $b$ in $\HH^1(\Gamma)\cap\HH^{1/2}_0(\Gamma)$, we have
$\inf_{\lambda_h\in\TT_h}\|b-\lambda_h\|_{\Gamma}\leq\const{}h\|b\|_{1,\Gamma}$. 
Setting $\lambda_h^{\ast}=\arginf_{\lambda_h\in\TT_h}\|b-\lambda_h\|_{\Gamma}$, it holds that
$\|\lambda_h^{\ast}\|_{1/2,\Gamma}\leq\const$ and $\|\lambda_h^{\ast}\|_{\mathcal{T}_h,\Gamma}\leq\const{}h^{-1/2}$.
From~\EQ{Chb1} we obtain $\|b\|_{\mathcal{T}_h,\Gamma}\leq\const{}h^{-1/2}$ and,
in turn, it follows from Theorem~\ref{thm:Strang2bound} that $|u-u_h|_{1,\Omega}\leq\const{}h^{-1/2}$ as $h\to{}0$.
Inserting the above estimates into~\EQ{Thm53}, we derive:
\begin{equation}
\label{eqn:L2h1/2}
\begin{aligned}
\|u-u_h\|_{\Omega}
&=
\sup_{\vp\in\LL^2(\Omega)\setminus\{0\}}\frac{|(u-u_h,\vp)_{\Omega}|}{\|\vp\|_{\Omega}}
\\
&\leq
\sup_{\vp\in\LL^2(\Omega)\setminus\{0\}}\frac{1}{\|\vp\|_{\Omega}}
\Big(\const{}h^{-1/2}h\|\vp\|_{\Omega}+\const{}h\|\vp\|_{\Omega}+\big(1+h^{-1/2}\big)h\|\vp\|_{\Omega}\Big)
\end{aligned}
\end{equation}
as $h\to{}0$. The estimate in~\EQ{CorL2global} then follows straightforwardly by combining terms.\hfill
\end{proof}
\begin{remark}
The reinforced regularity condition $b\in\HH^1(\Gamma)\cap\HH^{1/2}_0(\Gamma)$ on the slip distribution
ensures that $\inf_{\lambda_h\in\TT_h}\|b-\lambda_h\|_{\Gamma}\leq\const{}h$ as $h\to{}0$. Noting that
the weaker estimate $\inf_{\lambda_h\in\TT_h}\|b-\lambda_h\|_{\Gamma}\leq\const{}h^{1/2}$ suffices to obtain the result 
in Corollary~\ref{thm:CorL2global}, one is lead to question whether the reinforced regularity condition on~$b$ is
actually necessary. If the condition is dismissed, however, an interpolation estimate in the fractional
Sobolev space $\smash[t]{\HH^{1/2}_0(\Gamma)}$ is required. Interpolation estimates in fractional Sobolev spaces
are technical (see, for instance,~\cite{Dupont:1980fk}) and the particular result required here is to our knowledge not 
available. 
\end{remark}

It is noteworthy that the asymptotic convergence behavior according to Corollary~\ref{thm:CorL2global} is
consistent with the notion that the continuous WSM approximation incurs an $O(1)$ error, pointwise,
in the~$O(h)$ neighborhood of the discontinuity composed of the intersected elements as $h\to{}0$. 
In particular, denoting by $\Gamma_h$ the union of the elements for which the intersection of the fault 
with the closure of the element is non-empty, it holds that 
\begin{equation}
\|1\|_{\Gamma_h}=\big(\meas_N(\Gamma_h)\big)^{1/2}=O(h^{1/2})
\end{equation}
as $h\to{}0$, with $\meas_N(\Gamma_h)$ the $N$-Lebesgue measure of~$\Gamma_h$.

\subsection{Local approximation properties}
\label{sec:locconv}
Next, we consider the local approximation properties of WSM, i.e., on~$\Omega$ excluding a small neighborhood of
the fault. To this end, we will relate the WSM approximation to the standard Galerkin approximation associated with
a locally supported lift, outside the support of the lift.

For arbitrary $\varepsilon>0$, let $\varkappa^{\varepsilon}:=\{x\in\Omega:\mathrm{dist}(x,\varkappa)<\varepsilon\}$
denote the open
$\varepsilon$\nobreakdash-neighborhood of the dislocation fault. We consider a lift
$\lift{b}\in\HH^1_{0,\dirichlet}(\Omega\setminus\Gamma)$ of~$b$ with compact
support in $\varkappa^\varepsilon$. Let $\bar{u}{}_h\in\VV_h$ denote the Galerkin
approximation of the continuous complement of the solution with respect to the
jump lift $\lift{b}$ according to~\EQ{weak_varform1_fem}.
As a straightforward consequence of C\'ea's lemma, it follows that $\bar{u}_h+\ell_b$ is endowed with the quasi-optimal approximation
property:
\begin{equation}
\label{eqn:CeaLift}
\|u-(\bar{u}_h+\lift{b})\|_{1,\Omega\setminus\Gamma}\leq\const\inf_{v_h\in\VV_h}\|u-(v_h+\lift{b})\|_{1,\Omega\setminus\Gamma}
\end{equation}
A meaningful characterization of the local approximation properties of WSM is therefore provided by an estimate for the 
deviation $\|\bar{u}_h+\lift{b}-u_h\|_{1,\Omega\setminus\varkappa^{\varepsilon}}$ between the WSM approximation $u_h$ and the local-lift-based 
approximation $\bar{u}_h+\lift{b}$. Note that  $\varkappa^{\varepsilon}\supset\mathrm{supp}(\lift{b})$ implies that $\lift{b}$ vanishes 
on $\Omega\setminus\varkappa^{\varepsilon}$ and, hence, the estimate pertains to the deviation between $\bar{u}_h\in\VV_h$ and $u_h\in\VV_h$.

The characterization of the local approximation properties of WSM in the $\HH^1$-norm in Theorem~\ref{thm:CeaLocal} below is based on
the interpolation of a particular extension of functions in $\VV_h$ onto~$\varkappa^{\varepsilon}$. Specifically,
we consider an extension corresponding to the operator $E:\VV_h\to\HH^1(\Omega)$:
\begin{subequations}
\label{eqn:tildeE}
\begin{alignat}{2}
E(v_h)&=v_h&\qquad&\text{in }\Omega\setminus\varkappa^{\varepsilon}
\label{eqn:tildeEa}
\\
-\divstress\big(E(v_h)\big)&=0&\qquad&\text{in }\varkappa^{\varepsilon}
\label{eqn:tildeEb}
\end{alignat}
\end{subequations}
in the sense of distributions, and its optimal approximation in the norm defined by $a_{\Gamma}(\cdot,\cdot)$:
\begin{equation}
\label{eqn:Eh}
E_h(v_h)=\argmin_{w_h\in{}v_h+\VV_{h,\varkappa^{\varepsilon}}}a_{\Gamma}\big(E(v_h)-w_h,E(v_h)-w_h\big)
\end{equation}
where $\VV_{h,\varkappa^{\varepsilon}}:=\{v_h\in\VV_h:\mathrm{supp}(v_h)\subseteq\varkappa^{\varepsilon}\}\neq\emptyset$ denotes the class of approximation functions
of which the support is confined to~$\varkappa^{\varepsilon}$. Note that $\VV_{h,\varkappa^{\varepsilon}}\neq\emptyset$ implies $\varepsilon\geq{}\const{}h$.

Equations~\EQ{tildeEa} and~\EQ{Eh} imply that $E(v_h)$ and $E_h(v_h)$ coincide with~$v_h$ 
outside~$\varkappa^{\varepsilon}$.
Inside~$\varkappa^{\varepsilon}$, the extension~$E(v_h)$ is defined by the homogeneous elasticity problem~\EQ{tildeEb}. From $E(v_h)\in\HH^1(\Omega)$
it follows that $E(v_h)$ is continuous at $\partial\varkappa^{\varepsilon}$ in the trace sense, which implies that~\EQ{tildeEb}
is complemented by the Dirichlet boundary condition $E(v_h)=v_h$ on~$\partial\varkappa^{\varepsilon}$. 
By virtue of  the smoothness condition on the elasticity tensor in~\EQ{smoothelast}, the extension operator 
according to~\EQ{tildeE} exhibits an interior regularity property on~$\varkappa^{\varepsilon}$. In particular, for
all $v_h\in\VV_h$ it holds
that $E(v_h)\in\HH^2_{\mathrm{loc}}(\varkappa^{\varepsilon})$ (i.e., $\phi{}E(v_h)\in\HH^2(\varkappa^{\varepsilon})$ for any
$\phi\in{}C^{\infty}(\Omega)$ with compact support in~$\varkappa^{\varepsilon}$) and the following estimate holds
for each open subset $K\Subset\varkappa^{\varepsilon}$:
\begin{equation}
\label{eqn:InteriorRegularity}
\|E(v_h)\|_{2,K}\leq\const\|E(v_h)\|_{\varkappa^{\varepsilon}}
\qquad\forall{}v_h\in\VV_h;
\end{equation} 
see, for instance, \cite[\S{}6.3.1]{Evans:2009ph} or~\cite[\S{}2.3]{LionsMagenes1972I} for further details. 
It is important to note that estimate~\EQ{InteriorRegularity} in conjunction with the trace theorem implies that 
\begin{equation}
\label{eqn:RegularityTractionBound}
\big\|\mean{\sigma_{\nu}(E(v_h))}\big\|_{\varkappa}\leq\const\big\|E(v_h)\big\|_{1,\varkappa^{\varepsilon}}
\qquad\forall{}v_h\in\VV_h\,,
\end{equation}
for some $\const>0$ independent of~$v_h$, provided that there exists an open subset $K\Subset\varkappa^{\varepsilon}$ such 
that~$\varkappa\subset{}K$. This provision implies that the dislocation must be properly contained in~$\Omega$.
The proof of Theorem~\ref{thm:CeaLocal} involves an estimate of the difference between the average traction of $E(v_h)$
and its approximation $E_h(v_h)$ according to~\EQ{Eh}. In general, we can estimate $(b,\mean{\sigma_{\nu}(E(v_h)-E_h(v_h))})_{\Gamma}$
in the same manner as in the proof of~Theorem~\ref{thm:ANWSM}. The bound~\EQ{RegularityTractionBound} in combination with the
optimal-approximation property of~$E_h(v_h)$ in~\EQ{Eh} however suggests that in this case a sharper estimate can be established.
The derivation of such a refined estimate is intricate, however, as $(b,\mean{\sigma_{\nu}(\cdot)})_{\Gamma}$ is unbounded
on~$\HH^1(\Omega)$ and the estimate involves the difference 
between $E(v_h)$ and $E_h(v_h)$, which reside in different subspaces of~$\HH^1(\Omega)$. We therefore 
formulate the refined estimate in the form of a conjecture:
\begin{conjecture}
\label{thm:conjecture}
Assume that there exist an $\varepsilon^*>0$ and an open subset $K\Subset\varkappa^{\varepsilon}$ such that 
$\varkappa\subset{}K$. 
Consider a sequence of approximation spaces $\VV_{\mathcal{H}}$, the extension operator according to~\EQ{tildeE} and its approximation
according to~\EQ{Eh}. For all $\varepsilon\geq\varepsilon^*$ and all~$v_h\in\VV_h$, it holds that
\begin{equation}
\big\|\mean{\sigma_{\nu}(E(v_h)-E_h(v_h))}\big\|_{\varkappa}
\leq
\const
\big\|E(v_h)-E_h(v_h)\|_{1,\varkappa^{\varepsilon}}
\end{equation}
for some constant $\const>0$ independent of~$h$.
\end{conjecture}

Theorem~\ref{thm:CeaLocal} below presents a general characterization of the local approximation properties of WSM, independent of
Conjecture~\ref{thm:conjecture}, and a refinement, which is contingent on the conjecture. 

\begin{theorem}[Local approximation properties of WSM in the $\HH^1$-norm]
\label{thm:CeaLocal}
Assume that there exist an $\varepsilon^*>0$ and an open subset $K\Subset\varkappa^{\varepsilon}$ such that 
$\varkappa\subset{}K$. For arbitrary $\varepsilon\geq\varepsilon^*$, 
let $\lift{b}\in\HH^1_{0,\dirichlet}(\Omega\setminus\Gamma)$ denote a lift of~$b$ such that
$\supp(\lift{b})\subset\varkappa^{\varepsilon}$. Given a sequence of approximation spaces 
$\VV_{\mathcal{H}}$, for each $h\in\mathcal{H}$ let $\bar{u}_h$ denote the approximation of the
continuous complement corresponding to $\lift{b}$ in~\EQ{weak_varform1_fem} and let $u_h\in\VV_h$ denote the WSM 
approximation according to~\EQ{wsm}. Consider the extension operator in~\EQ{tildeE} and its approximation in~$\VV_h$
according to~\EQ{Eh}. It holds that
\begin{equation}
\label{eqn:CeaLocal}
\begin{aligned}
\|\bar{u}_h+\lift{b}-u_h\|_{1,\Omega\setminus\varkappa^{\varepsilon}}^2
&\leq{}
\const\|\lift{b}\|_{1,\varkappa^{\varepsilon}\setminus\Gamma}\|E(\bar{u}_h-u_h)-E_h(\bar{u}_h-u_h)\|_{1,\varkappa^{\varepsilon}}
\\
&\phantom{\leq}
+\const\big|\big(b,\mean{\sigma_{\nu}(E(\bar{u}_h-u_h)-E_h(\bar{u}_h-u_h))}\big)_{\Gamma}\big|
\end{aligned}
\end{equation}
For some constant $\const>0$ independent of~$h$ and~$\varepsilon$.
If, in addition, Conjecture~\ref{thm:conjecture} holds, then 
\begin{equation}
\label{eqn:CeaLocal2}
\|\bar{u}_h+\lift{b}-u_h\|_{1,\Omega\setminus\varkappa^{\varepsilon}}
\leq{}
\const\big(\|\lift{b}\|_{1,\varkappa^{\varepsilon}\setminus\Gamma}+\|b\|_{\Gamma}\big)
\sup_{v_h\in\VV_h}\inf_{w_h\in{}v_h+\VV_{h,\varkappa^{\varepsilon}}}\frac{\|E(v_h)-w_h\|_{1,\varkappa^{\varepsilon}}}
{\|v_h\|_{1,\Omega\setminus\varkappa^{\varepsilon}}}
\end{equation}
for some constant $\const>0$ independent of~$h$ and~$\varepsilon$.
\end{theorem}

\begin{proof}
We use the condensed notation $E_h:=E_h(\bar{u}_h-u_h)$.
Noting that~\EQ{Eh} implies that $E_h$ coincides with $\bar{u}_h-u_h$ on~$\Omega\setminus\varkappa^{\varepsilon}$ and
that~$\mathrm{supp}(\lift{b})\subset\varkappa^{\varepsilon}$, it holds that
\begin{equation}
\label{eqn:CeaLocal1a}
\|\bar{u}_h+\lift{b}-u_h\|_{1,\Omega\setminus\varkappa^{\varepsilon}}
=
\|\bar{u}_h-u_h\|_{1,\Omega\setminus\varkappa^{\varepsilon}}
\leq
\|E_h\|_{1,\Omega}
\end{equation}
Using the Poincar\'e inequality~\EQ{Poincare} and the strong positivity of~$a_{\Gamma}(\cdot,\cdot)$ according to~\EQ{strongpositive},
we obtain the following bound:
\begin{equation}
\label{eqn:localityseq}
\|E_h\|_{1,\Omega}^2
\leq
(1+C_p)\ucA^{-1}\big|a_{\Gamma}\big(E_h,E_h\big)\big|
\end{equation}
By introducing a suitable partition of zero, we derive from the WSM
formulation~\EQ{wsm} and the lift-based Galerkin
approximation~\EQ{weak_varform1_fem}:
\begin{equation}
\label{eqn:aGEE}
\begin{aligned}
a_{\Gamma}\big(E_h,E_h\big)
&=
a_{\Gamma}\big(\bar{u}_h-u_h,E_h)+a_{\Gamma}(E_h-(\bar{u}_h-u_h),E_h)
\\
&=
\big(b,\mean{\sigma_{\nu}(E_h)}\big)_{\Gamma}-a_{\Gamma}(\lift{b},E_h)+a_{\Gamma}(E_h-(\bar{u}_h-u_h),E_h)
\end{aligned}
\end{equation}
By virtue of~\EQ{tildeEb}, $a_{\Gamma}(v_h,E(\bar{u}_h-u_h))$ vanishes for all~$v_h\in\VV_{h,\varkappa^{\varepsilon}}$. 
The optimality condition associated with~\EQ{Eh} therefore reduces to:
\begin{equation}
\label{eqn:OptimalityEh}
a_{\Gamma}(v_h,E_h)=0
\qquad\forall{}v_h\in\VV_{h,\varkappa^{\varepsilon}}\,.
\end{equation} 
Noting that~$E_h-(\bar{u}_h-u_h)\in\VV_{h,\varkappa^{\varepsilon}}$, Equation~\EQ{OptimalityEh} implies that
the right-most term in the ultimate expression in~\EQ{aGEE} vanishes.
By virtue of the interior regularity of $E(\bar{u}_h-u_h)$, the following identity holds:
\begin{equation}
\label{eqn:aliftb}
a_{\Gamma}\big(\lift{b},E(\bar{u}_h-u_h)\big)
=
\big(b,\mean{\sigma_{\nu}(E(\bar{u}_h-u_h))}\big)_{\Gamma}
\end{equation}
Let us note that~\EQ{aliftb} corresponds to the (admissible) identification of the duality pairing between $b\in\HH^{1/2}_0(\varkappa)$ and 
$\mean{\sigma_{\nu}(E(\bar{u}_h-u_h))}\in\HH^{-1/2}(\varkappa)$ to an $\LL^2$ inner product. From~\EQ{aGEE}\nobreakdash--\EQ{aliftb} and
the triangle inequality we then deduce that:
\begin{equation}
\label{eqn:finalCeaLocal}
\big|a_{\Gamma}(E_h,E_h)\big|
\leq
\big|a_{\Gamma}(\lift{b},E(\bar{u}_h-u_h)-E_h)\big|
+
\big|(b,\mean{\sigma_{\nu}(E(\bar{u}_h-u_h)-E_h)})_{\Gamma}\big|
\end{equation}
We recall that $E(\bar{u}_h-u_h)-E_h$ vanishes on~$\Omega\setminus\varkappa^{\varepsilon}$.
Estimate~\EQ{CeaLocal} then follows straightforwardly from~\EQ{CeaLocal1a}, \EQ{localityseq}, \EQ{finalCeaLocal} and 
the continuity of $a_{\Gamma}(\cdot,\cdot)$ according to~\EQ{aGcontinuous}.

To prove the auxiliary assertion~\EQ{CeaLocal2}, we note that $\|E(v_h)-E_h(v_h)\|_{1,\Omega}$ in fact depends only on the trace of $v_h$ on $\partial\varkappa^{\varepsilon}$
and, by the trace theorem, it holds that~$\|E(v_h)-E_h(v_h)\|_{1,\Omega}\leq\const\|v_h\|_{1,\Omega\setminus\varkappa^{\varepsilon}}$.
Therefore, if Conjecture~\ref{thm:conjecture} holds, we obtain from~\EQ{CeaLocal}:
\begin{equation}
\label{eqn:finalCeaLocal2}
\begin{aligned}
\|\bar{u}_h-u_h\|_{1,\Omega\setminus\varkappa^{\varepsilon}}^2
&\leq
\const\big(\|\lift{b}\|_{1,\varkappa^{\varepsilon}\setminus\Gamma}+\|b\|_{\Gamma}\big)\big\|E(\bar{u}_h-u_h)-E_h(\bar{u}_h-u_h)\big\|_{1,\varkappa^{\varepsilon}}
\\
&\leq
\const\big(\|\lift{b}\|_{1,\varkappa^{\varepsilon}\setminus\Gamma}+\|b\|_{\Gamma}\big)
\sup_{v_h\in\VV_h}\frac{\|E(v_h)-E_h(v_h)\|_{\varkappa^{\varepsilon}}}{\|v_h\|_{1,\Omega\setminus\varkappa^{\varepsilon}}}
\|\bar{u}_h-u_h\|_{1,\Omega\setminus\varkappa^{\varepsilon}}
\end{aligned}
\end{equation}
Estimate~\EQ{CeaLocal2} follows directly from the identity in~\EQ{CeaLocal1a}, the ultimate bound in~\EQ{finalCeaLocal2} and the definition 
of~$E_h$ in~\EQ{Eh}.
\end{proof}

Theorem~\ref{thm:CeaLocal} essentially implies that if the lift-based approximation~\EQ{weak_varform1_fem} displays optimal global convergence,
then the WSM approximation~\EQ{wsm} displays optimal local convergence.
\begin{corollary}
\label{thm:OptimalLocal}
Assume that Conjecture~\ref{thm:conjecture} holds and that
there exist an $\varepsilon^*>0$ and an open subset $K\Subset\varkappa^{\varepsilon^*}$ such that $\varkappa\subset{}K$. 
Assume that $A_{ijkl}\in{}C^{k+1}(\overline{\Omega})$ for some integer $k\geq{}0$ and that $\Omega$ is convex or of class $C^{k+2}$.
Assume that $b$ admits a sufficiently smooth local lifting, in particular, that there exists 
an $\lift{b}$ 
such that $\mean{\sigma_{\nu}(\lift{b})}=0$, $\divstress(\lift{b})\in{}\HH^k(\Omega)$ and $\mathrm{supp}(\lift{b})\subset\varkappa^{\varepsilon}$.
Let $\VV_{\mathcal{H}}$ denote a sequence of $\HH^1(\Omega)$-conforming piecewise polynomial approximation spaces of degree $p\geq{}1$ with 
approximation property~\EQ{OptimalInterpolation}. 
Let $u$ denote the solution to the Volterra dislocation problem~\EQ{weak_varform} and let $u_{\mathcal{H}}$ denote the sequence of WSM 
approximations~\EQ{wsm} corresponding to the approximation spaces $\VV_{\mathcal{H}}$. It holds that
\begin{equation}
\label{eqn:OptimalLocal}
\|u-u_h\|_{1,\Omega\setminus\varkappa^{\varepsilon}}\leq\const{}h^l
\end{equation}
as $h\to{}0$ with $\const>0$ independent of~$h$ and $l=\min\{p,k+1\}$.
\end{corollary}  
\begin{proof}
Note that the following integration-by-parts identity holds for all $v\in\HH^1(\Omega)$:
\begin{equation}
\label{eqn:OptLoc1}
a_{\Gamma}(\lift{b},v)=\int_{\partial(\Omega\setminus\Gamma)}v\cdot\sigma_n(\lift{b})-\int_{\Omega\setminus\Gamma}v\cdot\divstress(\lift{b})
=\int_{\Gamma}v\cdot\mean{\sigma_{\nu}(\lift{b})}-\int_{\Omega\setminus\Gamma}v\cdot\divstress(\lift{b})
\end{equation}
The second identity follows by rearranging terms and applying~\EQ{rearrangement}. The first term in the ultimate expression in~\EQ{OptLoc1} vanishes by
virtue of the conditions on~$\lift{b}$. Moreover, because $l(v)=(f,v)=0$, the right member of~\EQ{weak_varform1} corresponds
to 
\begin{equation}
l(v)-a_{\Gamma}(\lift{b},v)=-\int_{\Omega\setminus\Gamma}v\cdot\divstress(\lift{b})
\end{equation}
By virtue of $\divstress(\lift{b})\in\HH^k(\Omega)$ and the conditions on the elasticity tensor and the domain, it then holds that the
continuous complement $\bar{u}$ in~\EQ{weak_varform1} resides in $\HH^{k+2}(\Omega)$; see, for instance, \cite[Theorem~6.3.5]{Evans:2009ph}.
The quasi-optimality of the Galerkin approximation~\EQ{weak_varform1_fem} in combination with the optimal approximation properties of~$\VV_h$ according to~\EQ{OptimalInterpolation}
implies that~$\|\bar{u}-\bar{u}_h\|_{1,\Omega}\leq\const{}h^l$. Moreover, the conditions on the elasticity tensor imply that
$E(v_h)\in\HH^{k+2}_{\mathrm{loc}}(\varkappa^{\varepsilon})$ and, in turn,~\EQ{OptimalInterpolation} yields:
\begin{equation}
\sup_{v_h\in\VV_h}\inf_{w_h\in{}v_h+\VV_{h,\varkappa^{\varepsilon}}}\frac{\|E(v_h)-w_h\|_{1,\varkappa^{\varepsilon}}}{\|v_h\|_{1,\Omega\setminus\varkappa^{\varepsilon}}}
\leq\const{}h^l
\end{equation}
The estimate~\EQ{OptimalLocal} then follows from 
\begin{equation}
\|u-u_h\|_{1,\Omega\setminus\varkappa^{\varepsilon}}
\leq\|\bar{u}-\bar{u}_h\|_{1,\Omega}+\|\bar{u}_h+\lift{b}-u_h\|_{1,\Omega\setminus\varkappa^{\varepsilon}}
\end{equation}
and Theorem~\ref{thm:CeaLocal}.
\end{proof}

\section{Numerical results}
\label{sec:results}

We will assess the approximation properties of the Weakly-enforced Slip Method
on the basis of three different test cases. All three test cases are designed
to have analytical results available. This allows us to compare the exact
solution of boundary value problem \eqref{eqn:BVP} with the WSM
solution \eqref{eqn:wsm} and study the behavior of errors and convergence under
refinement of the finite finite-element mesh.

The considered test cases are:
\begin{enumerate}
  \item[I.] a two-dimensional infinite domain loaded in plane strain, with
  a straight, finite dislocation and smooth slip distribution,
  \item[II.] a three-dimensional semi-infinite domain with traction-free
  surface, a planar, non-rupturing dislocation and constant slip, and
  \item[III.] a three-dimensional semi-infinite domain with traction-free
  surface, a planar, surface rupturing dislocation and constant slip.
\end{enumerate}

All three test cases are in principle set on infinite domains. To make
the problems amenable to treatment by the finite-element method, we
truncate the domains, and restrict the analyses to suitably large, but finite,
neighbourhoods of the dislocation. In order to focus on the treatment of the
dislocations, boundary truncation errors are controlled by constraining the
finite-element approximation to the analytical solution at the artificial
lateral boundaries.

Let us note that of the three test cases, only the first one is by
design in full accordance with the theory developed in
Section~\ref{sec:convergence}. Test case II and III feature a piecewise
constant slip to match available analytical solutions, violating the
condition of Lemma~\ref{thm:continuity} which states that the slip can
be lifted into $ \HH^1_{0,\dirichlet}(\Omega\setminus\Gamma) $, i.e., $ b \in \smash[t]{\HH_0^{1/2}( \varkappa )} $. Test
case III moreover considers a rupturing fault and, accordingly, it has 
nonzero slip at the intersection with the domain boundary. The results below
convey that the main results of the theory nonetheless uphold, 
suggesting that the conditions under which they apply can be relaxed.

\subsection{Test case I: 2D plane strain}

The first test case is a line dislocation in plane strain in an infinite two-dimensional isotropic domain. 
The dislocation is straight and of unit length, with a smoothly varying slip that is
tangent to the dislocation line, making it a pure shear dislocation.
Figure~\ref{fig:testcase1} shows a schematic of the computational setup.
\begin{figure}
  \centering
  \includegraphics{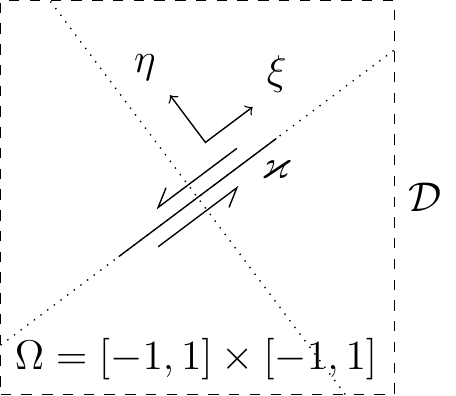}
  \caption{%
    Schematic overview of test case I. The dashed line marks the
    computational domain, which is a truncation of the actual
    (infinite) domain. A unit-length dislocation $\varkappa$ is placed at an
    $\mathrm{arctan}(3/4)$ angle with the horizontal axis at the center of the
    computational domain.}
  \label{fig:testcase1}
\end{figure}

Introducing an arclength coordinate $\xi$ along the fault, we consider a smooth, piecewise quadratic slip distribution according to:
\begin{equation}
b : \mathbb R \rightarrow b_0 \begin{cases}
\tfrac 3 2 + 6 \xi + 6 \xi^2 & -\tfrac 1 2 \hfill < \xi < -\tfrac 1 6 \\
1 - 12 \xi^2 & -\tfrac 1 6 < \xi < \hfill \tfrac 1 6 \\
\tfrac 3 2 - 6 \xi + 6 \xi^2 & \hfill \tfrac 1 6 < \xi < \hfill \tfrac 1 2 \\
0 & \text{otherwise}
\end{cases}
\label{eqn:smoothslip}
\end{equation}
The dislocation, which corresponds to the support of the slip, is located in the interval $\xi\in[-\tfrac{1}{2},\tfrac{1}{2}]$.
The slip is symmetric, with zero displacement at the tips and smoothly opening
and closing. We set the scaling factor $b_0=0.1$. The value of $b_0$ is however non-essential
on account of the linearity of the problem. An analytical solution to this
problem can be constructed for an infinite, homogeneous, isotropic domain, by adapting
the well known solution for edge dislocations~\cite{hirth82}. Introducing a perpendicular 
coordinate $ \eta $, the resulting displacement along the $(\xi,\eta)$ coordinates is expressed 
in terms of Lam\'e parameters $ \lambda $ and $ \mu $ as
\begin{equation}
u( \xi, \eta ) = b_0 \big[ 6 U(\xi-\tfrac 1 2,\eta) - 18 U(\xi-\tfrac 1
6,\eta) + 18 U(\xi+\tfrac 1 6,\eta) - 6 U(\xi+\tfrac 1 2,\eta) \big]
\label{eqn:testcase1exact}
\end{equation}
where
\begin{displaymath}
U(\xi,\eta) = \begin{pmatrix}
2 \frac {2\lambda+3\mu} {\lambda+2\mu} \xi \eta \\ -\frac \mu
{\lambda+2\mu} \xi^2 - \tfrac {2\lambda+\mu} {\lambda+2\mu} \eta^2
\end{pmatrix} \log \sqrt{ \xi^2 + \eta^2 } + \begin{pmatrix}
\frac {3\lambda+4\mu} {\lambda+2\mu} \eta^2 - \xi^2 \\ 2 \frac
{\lambda} {\lambda+2\mu} \xi \mu \end{pmatrix} \arctan{ \frac
\xi \mu }.
\end{displaymath}
A finite sized problem is obtained by truncating the domain to $
\Omega = (-1,1)^2 $ and introducing a Dirichlet condition
at the boundary $\partial\Omega= \dirichlet$ with data corresponding to
the exact solution. The fault is located in the center of the domain at an
angle of $\mathrm{arctan}(3/4)$. All numerical results are generated for Lam\'e
parameters $ \lambda = 1, \mu = 1 $.

Figure~\ref{fig:disptestcase1} shows the exact solution of
Eqn.~\eqref{eqn:testcase1exact} (left) in the deformed configuration, side by side with the WSM
approximation for linear shape functions on a $16 \times 16$ mesh
(right). The colors indicate displacement magnitude on a log scale. The
grid-line pattern represents the distortion of a mesh that is uniform in the undeformed 
configuration. The lines coincide with element
edges for the finite element computation on the right, and are matched
on the left to facilitate visual comparison. Because the test case is
symmetric, we expect the images to be approximately rotationally
symmetric. We observe that the rotational symmetry breaks at the dislocation, 
where the exact solution is discontinuous, while the WSM approximation
is continuous throughout the domain.

\begin{figure}
  \centering
  \includegraphics{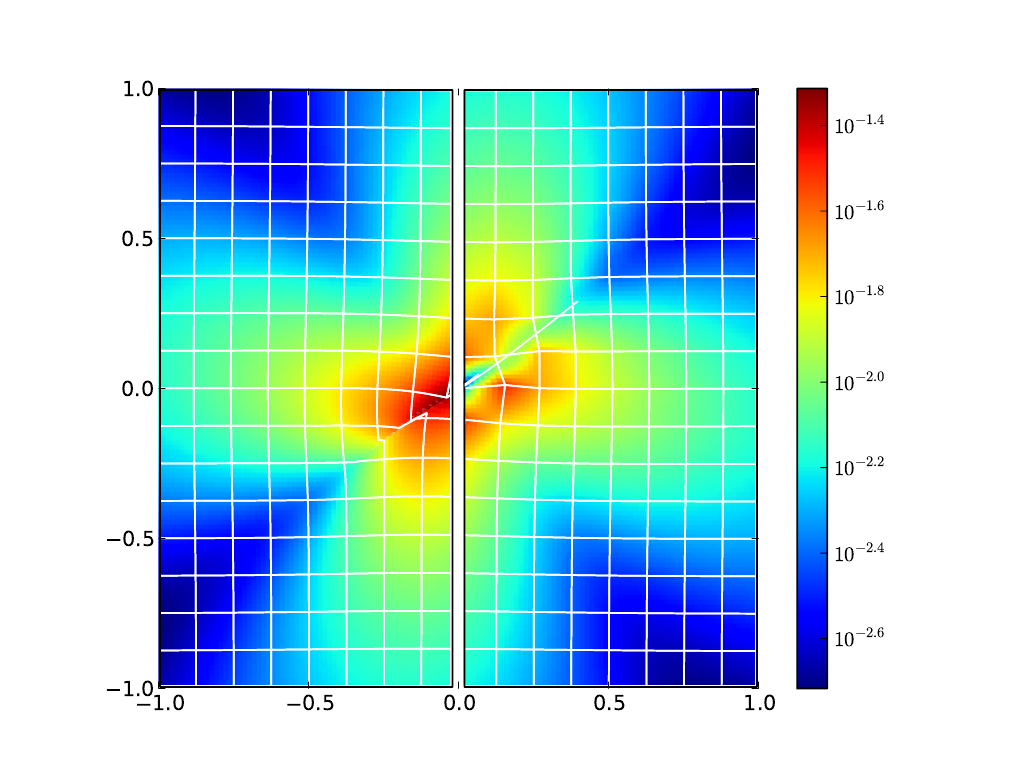}
  \caption{%
    Side by side view of the discontinuous exact solution of test case I
    (left) with the continuous WSM approximation on a $16 \times 16$
    mesh and linear shape functions (right). Color
    represents displacement magnitude on a log scale. Mesh lines on the
    left are introduced to indicate displacements and to facilitate
    comparison.}
  \label{fig:disptestcase1}
\end{figure}

The global symmetry of the displacement pattern of
Figure~\ref{fig:disptestcase1} indicates that the error induced by the WSM approximation
diminish with distance from the dislocation. To quantify this
observation, Figure~\ref{fig:disperrtestcase1} shows the errors for a $
16 \times 16 $ mesh (left), and for a sequence of meshes along
the line B--B' perpendicular to the fault (right). The latter shows that for every mesh the
error decays exponentially with increasing distance from the fault, with
the exception of a narrow zone in the vicinity of the
dislocation, where the error displays a transition to the 
constant error $ \tfrac 1 2 b_0 $ at the dislocation. One may observe that the
error at the dislocation is independent of element size. Away from the dislocation, errors are seen to
decrease with element size, approximately reducing by $10^{0.6} \approx
4$ whenever the mesh width is halved. This error reduction provides a first indication 
that local $\LL^2$ convergence is optimal at $ O( h^{p+1} )$ as $h\to{}0$. 

\begin{figure}
  \centering
  \includegraphics{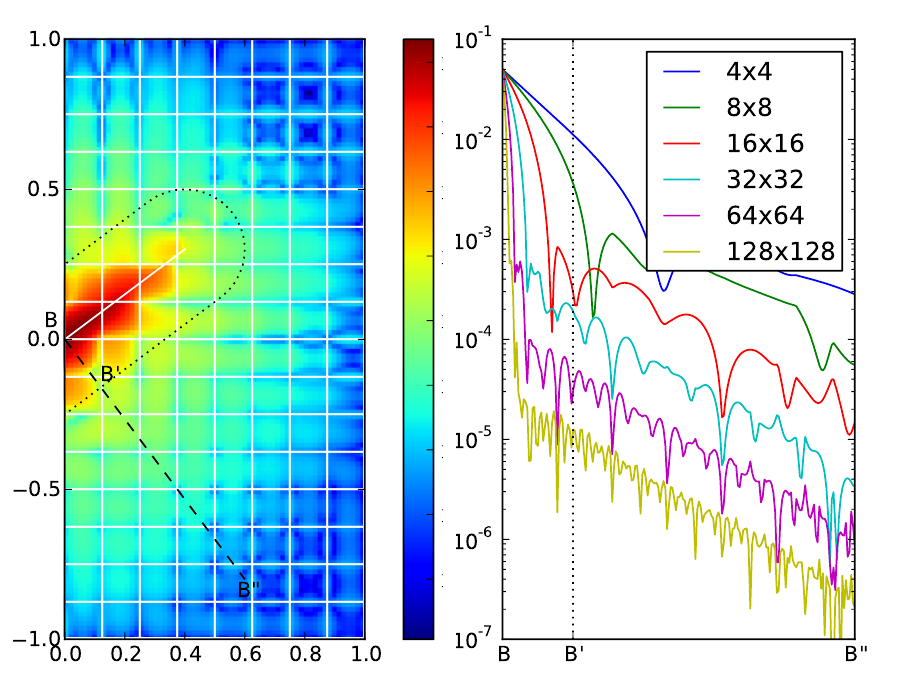}
  \caption{%
    Absolute value of the displacement error for the $16 \times 16$ mesh
    result of Figure~\ref{fig:disptestcase1} (left), side by side with
    error evaluations along the line B--B'' for a sequence of meshes of
    increasing density. The errors shown on the left thus correspond
    with the 3rd (red) curve on the right.}
  \label{fig:disperrtestcase1}
\end{figure}

Convergence of the WSM approximation under mesh refinement is further examined in Figure~\ref{fig:convtestcase1}, which
displays the $\LL^2$-norm and $\HH^1$-norm of the error for both linear ($p=1$) and quadratic ($p=2$) shape
functions. The curve marked `global' shows the error integrated over the entire
computational domain. The curve marked `local' shows the error integrated over
the domain excluding an $0.1$-neighborhood of the dislocation, corresponding to the dotted
area in Figure~\ref{fig:disperrtestcase1}. We observe that for both linear and quadratic approximations,
the global $\HH^1$-norm of the error diverges as $O(h^{-1/2})$ as $h\to{}0$, while
the global $\LL^2$-norm of the error converges as $O(h^{-1/2})$, independent of the order of approximation. 
This asymptotic behavior is
in accordance with the estimates in Corollaries~\ref{thm:CorH1global} and~\ref{thm:CorL2global}.
Figure~\ref{fig:convtestcase1} moreover corroborates that the local $\HH^1$-norm of the error converges
as $O(h^p)$ as $h\to0$, in agreement with the estimate in Corollary~\ref{thm:OptimalLocal}.
The local $ \LL^2 $-norm of the error, for which no theory was developed, also displays an optimal convergence rate of $O(h^{p+1})$ as $h\to{}0$.

\begin{figure}
  \centering
  \includegraphics{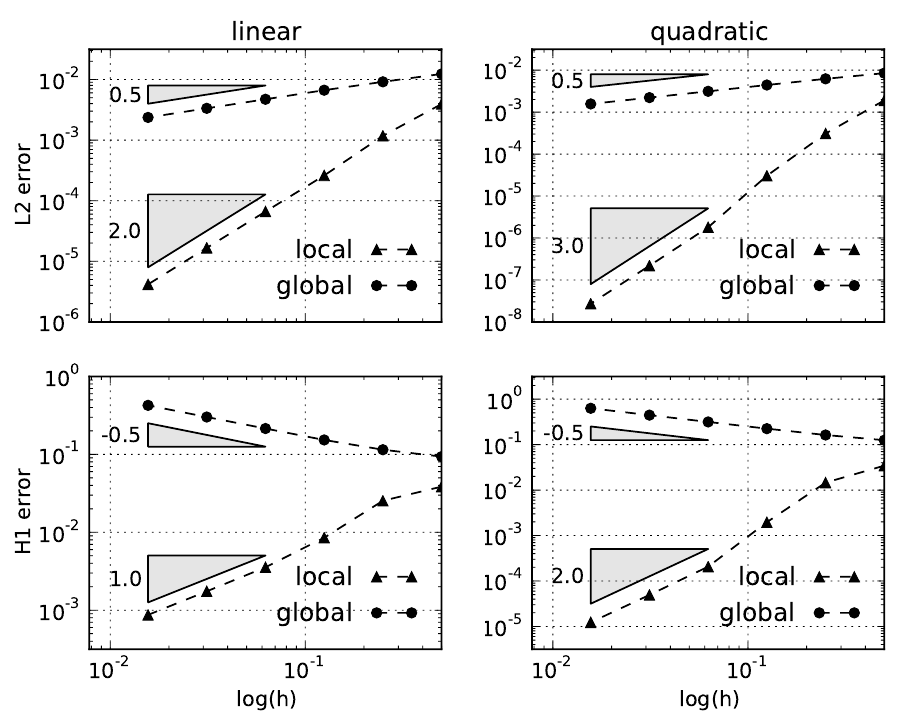}
  \caption{%
    Mesh convergence of WSM applied to test case I, showing the $\LL^2$-norm
  (top) and $\HH^1$-norm (bottom) of the error for linear (left) and
    quadratic (right) shape functions. The markers corresponds to mesh sequences of $\{4\times{}4,8\times8,\ldots,128\times128\}$
	elements are considered. The global error is computed by
    integration over the entire computational domain $ \Omega $, the
    local error by integration over the $>$0.1 distance exterior around
    $ \varkappa $, bounded by the dotted line in
    Figure~\ref{fig:disperrtestcase1}.}
  \label{fig:convtestcase1}
\end{figure}

\subsection{Test case II: 3D traction-free halfspace}

The second test case is a planar dislocation buried in a semi-infinite, three dimensional,
homogeneous, isotropic domain with a flat, traction-free surface.
The dislocation is a rectangular plane, the sides of which are displaced over a
constant distance in both strike and in dip direction, such that in geodetic
terms the setting is that of an oblique left-lateral thrust fault.
Figure~\ref{fig:testcase2} shows a schematic of the computational setup.
Analytical solutions to this problem have been derived by Okada
\cite{okada85,okada92}. We use homogeneous Lam\'e constants $ \lambda
= 1, \mu = 1 $ for all subsequent computations.

\begin{figure}
  \centering
  \includegraphics{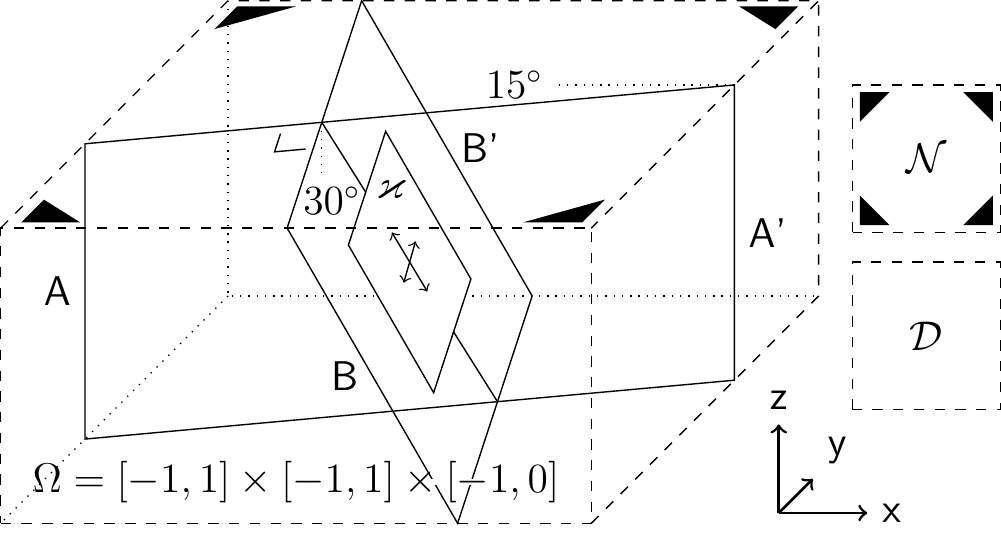}
  \caption{%
    Schematic overview of test case II. A $ 1 \times 3^{-1/2} $ sized
    dislocation plane is positioned at $ 15^\circ $ strike and $
    30^\circ $ dip, spanning a $0.25$--$0.75$ depth range. The surface is
    traction free. Dashed faces mark the computational domain, which
    is a truncation of the actual (semi-infinite) domain. The fault is
    represented by the B--B' plane, where at $ \varkappa $ the medium
    is dislocated by a constant $0.2$ displacement in strike direction
    and a constant $0.1$ displacement in dip direction. The perpendicular
    intersection plane A--A' serves visualization purposes only; see
    Figures~\ref{fig:disptestcase2} and~\ref{fig:disptestcase3}.}
  \label{fig:testcase2}
\end{figure}

\begin{figure}
  \includegraphics[width=\textwidth]{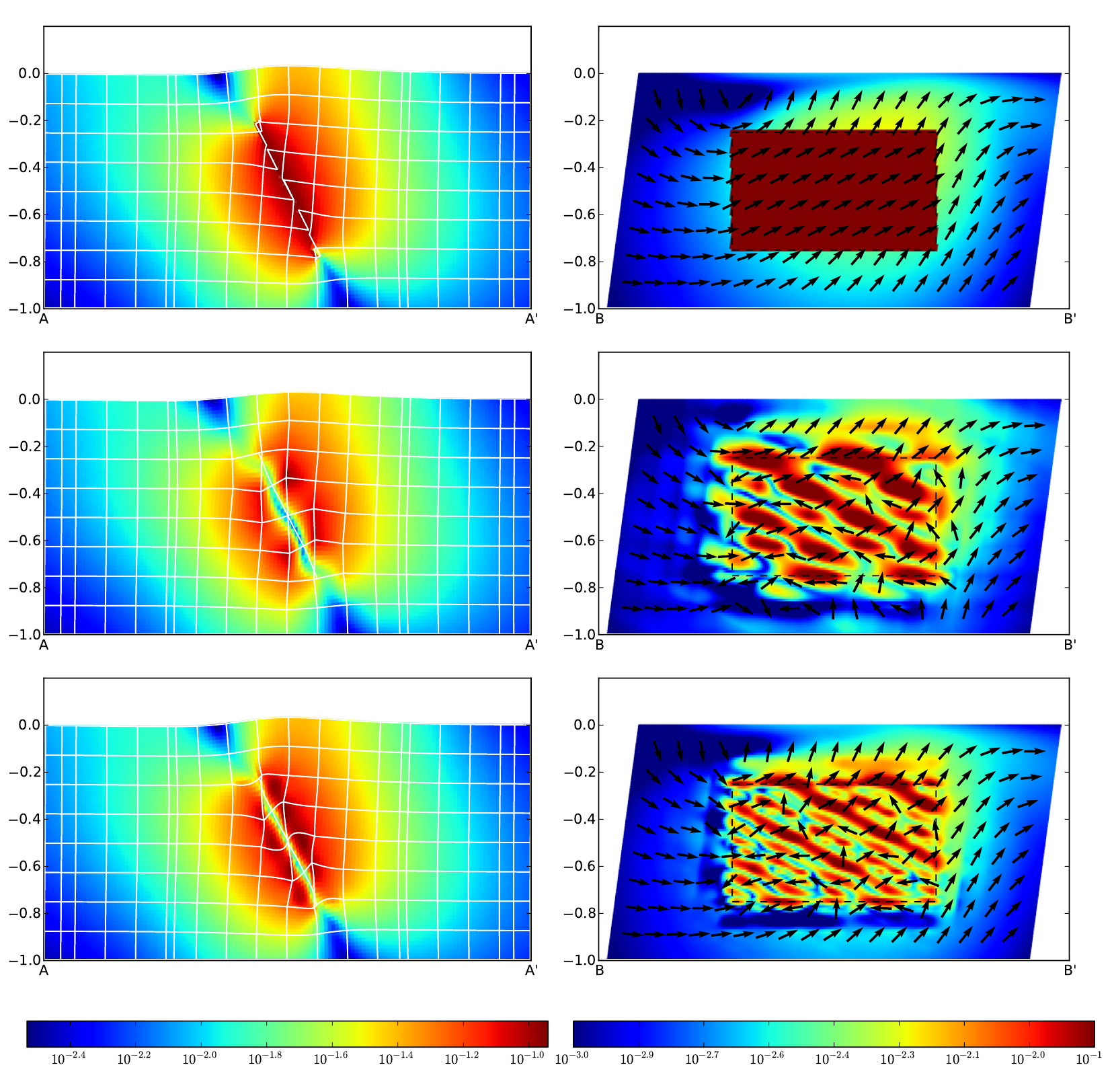}
  \caption{%
    Cross-section views of the displacement field of test case II,
    corresponding to the A-A' and B-B' intersection planes of
    Figure~\ref{fig:testcase2}: exact displacement as derived by
    Okada~\cite{okada92} (top) and WSM approximation on a $16 \times 16
    \times 8$ mesh with linear (middle) and quadratic shape
    functions (bottom). Colors indicate the displacement magnitude on a
    logarithmic scale analogous to Figure~\ref{fig:disptestcase1}.
    Arrows indicate direction only.}
  \label{fig:disptestcase2}
\end{figure}

Computations are performed on a truncated domain $ \Omega = [-1,1]
\times [-1,1] \times [-1,0] $, with Dirichlet conditions enforcing the
Okada solution at the five truncation planes.
Figure~\ref{fig:disptestcase2} shows displacements along the
intersection planes A-A' and B-B' indicated in
Figure~\ref{fig:testcase2}. The topmost figure shows the exact solution
according to Okada's equations. The middle and bottom figures display the
WSM approximation of the displacement on a $16 \times 16 \times 8$ mesh, for
linear and quadratic shape functions, respectively. The
displacement fields in the WSM approximations are seen to be continuous everywhere, though they
become highly irregular at the dislocation. At further distances
from the dislocation, however, the approximations exhibit very good agreement
with the exact solution, despite the coarseness of the considered mesh.

Figure~\ref{fig:convtestcase2} examines convergence of the global
and local norms of the error under mesh refinement. The results confirm the $O(h^{-1/2})$ divergence
and the $O(h^{1/2})$ convergence of the global $\HH^1$-norm and global $\LL^2$-norm,
respectively, in agreement with the estimates in Corollaries~\ref{thm:CorH1global} and~\ref{thm:CorL2global}.
Furthermore, we observe $O(h^{p})$ convergence of the local $\HH^1$-norm
in accordance with the estimate in Corollary~\ref{thm:OptimalLocal} and
$O(h^{p+1})$ convergence for the local $\LL^2$-norm. Figure~\ref{fig:convtestcase2}
moreover displays the $ \LL^2 $-norm of the displacement error at the surface,
which exhibits an optimal convergence rate of $O(h^{p+1})$. It is to be noted
that the agreement of the observed convergence rates for the global $\HH^1$-norm and $\LL^2$-norm 
and the local $\HH^1$-norm with the estimates in Section~\ref{sec:convergence}
is non-obvious, as a piecewise constant slip cannot be lifted
into $\HH_{0,\dirichlet}^1(\Omega\setminus\Gamma)$, and the test case under consideration therefore fails
to meet the conditions imposed in Section~\ref{sec:convergence}.
\begin{figure}
  \centering
  \includegraphics{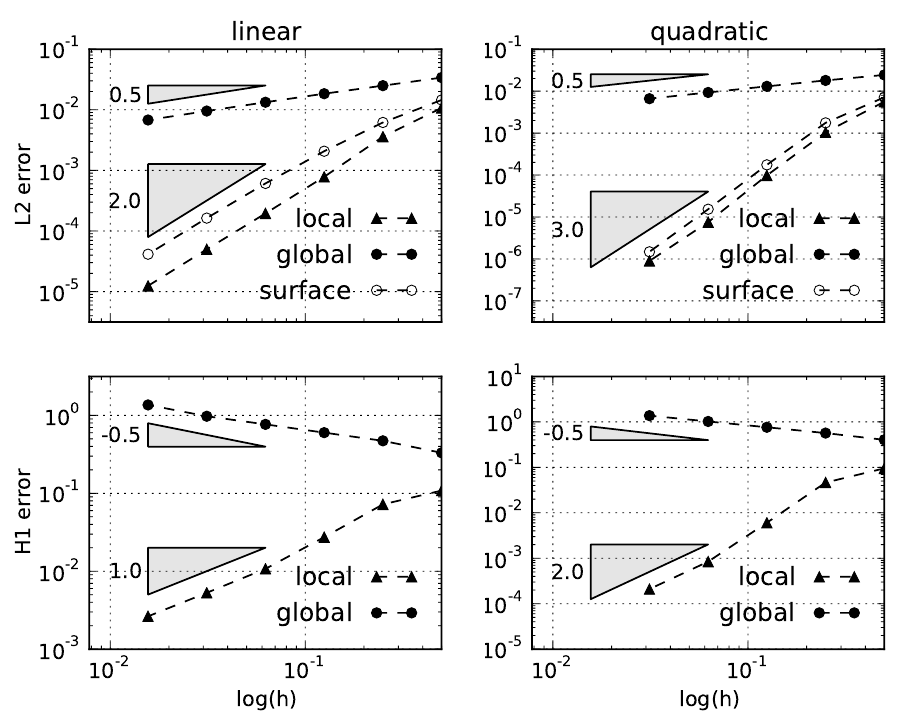}
  \caption{%
    Mesh convergence of WSM for test case II. The figures are
    analogous to Figure~\ref{fig:convtestcase1}, with the addition of
    the $\LL^2$-norm of the error in the displacement at the traction-free surface.}
  \label{fig:convtestcase2}
\end{figure}

\subsection{Test case III: 3D traction-free rupturing halfspace}

\begin{figure}
  \centering
  \includegraphics{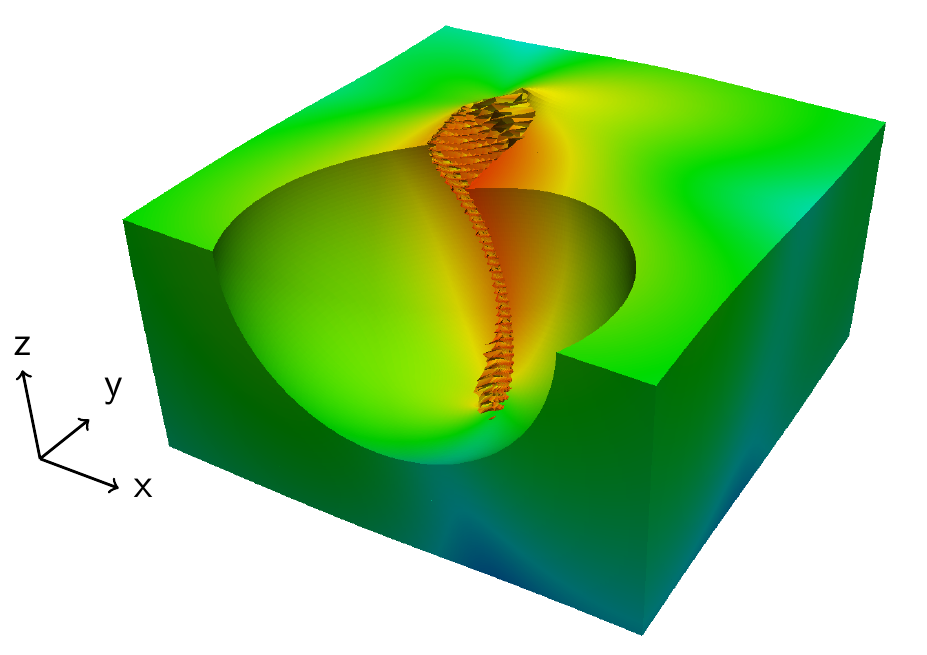}
  \caption{%
    Perspective projection of a WSM computation of test case III, using linear
    shape functions on a $128 \times 128 \times 64$ mesh. Deformations are
    amplified by a factor 2 for visualization purposes; colors represent displacement
    magnitude on a log scale. The orientation aligns roughly with that of
    Figure~\ref{fig:testcase2}. A spherical cut-out exposes part of the
    interior of the domain, revealing a zone of large displacements local to
    the dislocation.}
  \label{fig:rupture3d}
\end{figure}

The third test case is in everything equal to the second, except that
the dislocation is now extended in vertical direction to span a depth range
of 0--0.75, and the dip slip direction is reversed to form the
geodetic equivalent of a rupturing reverse fault; see
Figure~\ref{fig:rupture3d}. Figure~\ref{fig:disptestcase3} shows the
cross section displacements, and Figure~\ref{fig:convtestcase3} the
norms of the displacement-error under mesh refinement. We observe that
the global (resp. local) $\HH^1$-norm and $\LL^2$-norm of the error
are again proportional to~$h^{-1/2}$ and~$h^{1/2}$ (resp.~$h^p$ and~$h^{p+1}$), respectively.
In addition to the $\LL^2$-norm of
the error in the displacement field at the traction-free boundary, 
Figure~\ref{fig:convtestcase3} also presents the local $\LL^2$\nobreakdash-norm
of the error at the traction-free boundary, i.e., the error on the surface excluding an $0.1$-neighborhood of
the dislocation. 
The global $\LL^2$-norm of the surface error converges
with an asymptotic rate of $O( h^{1/2} )$. This suboptimal convergence behavior is
caused by the fact that in this case the discontinuity reaches the
surface. The local $\LL^2$-norm of the error at the surface again display optimal
convergence at a rate of $O(h^{p+1})$. It is to be noted that
the convergence results for the
global $\HH^1$-norm and $\LL^2$-norm and the local $\HH^1$-norm agree
with the estimates in Section~\ref{sec:convergence}, despite the fact that
the analysis in Section~\ref{sec:convergence} is restricted to non-rupturing
faults.

\begin{figure}
  \includegraphics[width=\textwidth]{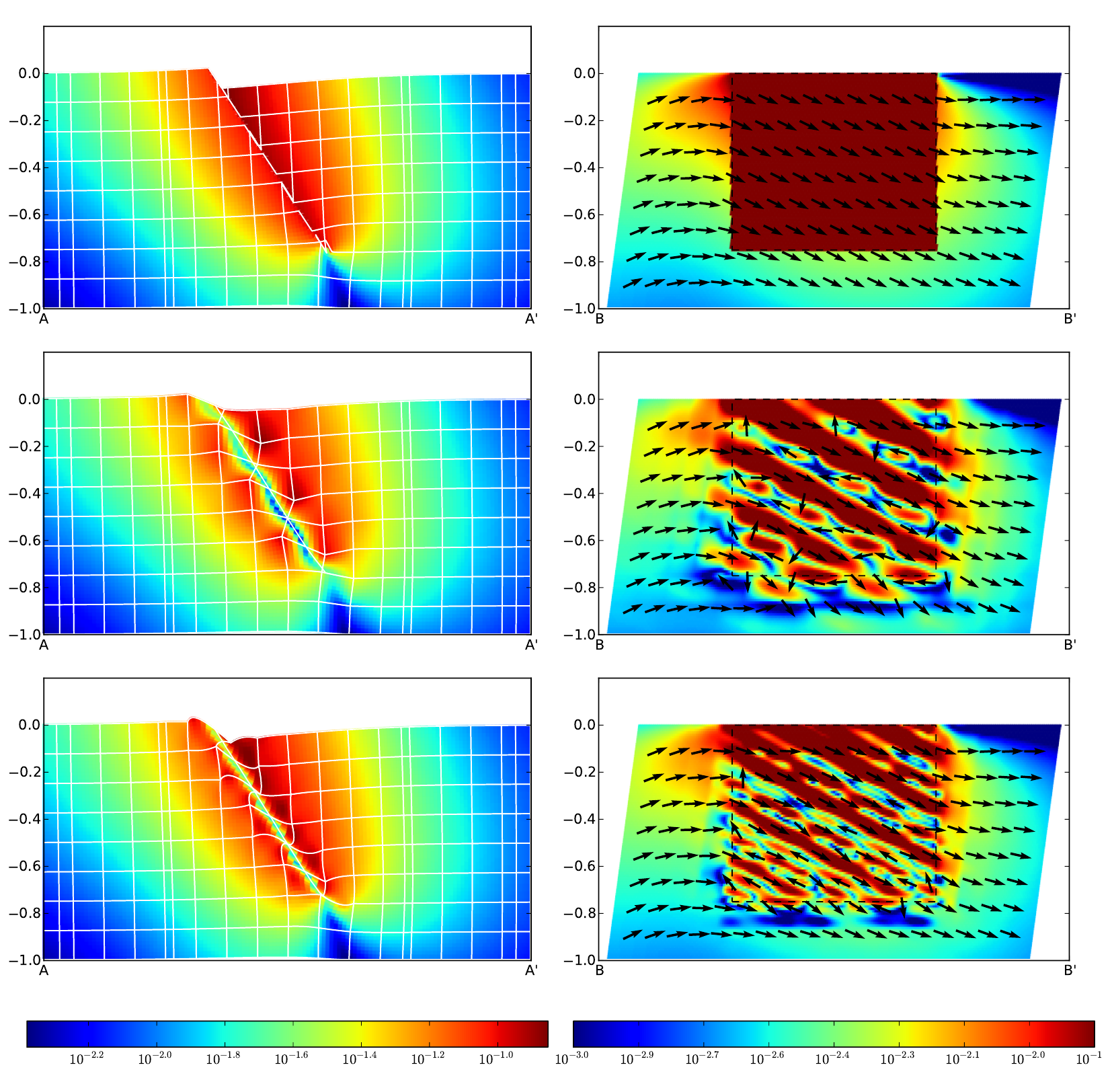}
  \caption{%
    Cross-section views of the displacement field of test case III. The figures
    are analogous to Figure~\ref{fig:convtestcase1}.}
  \label{fig:disptestcase3}
\end{figure}

\begin{figure}
  \centering
  \includegraphics{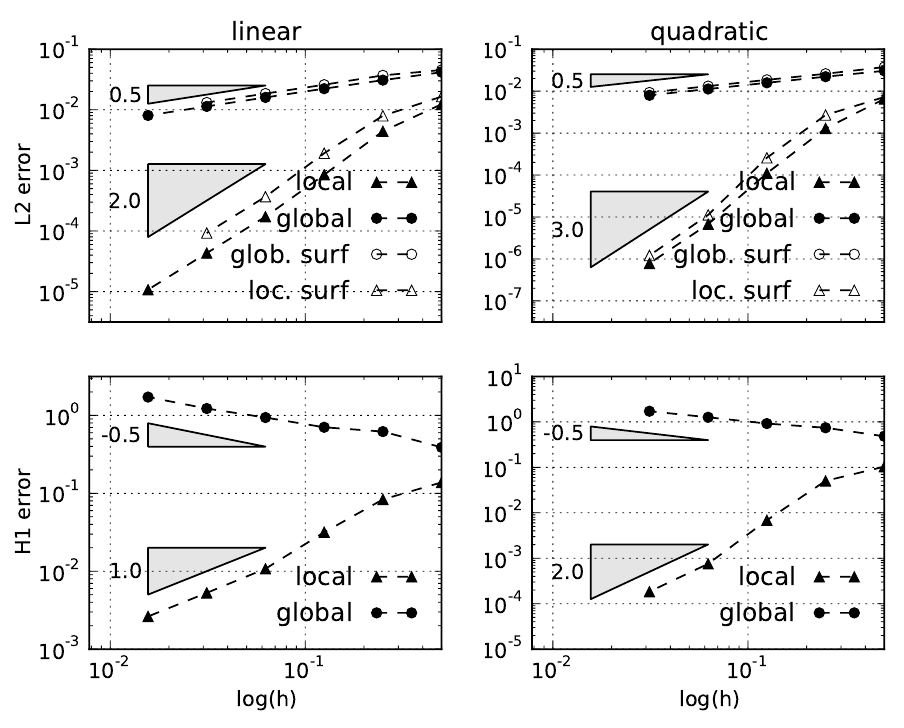}
  \caption{%
    Mesh convergence of WSM applied to test case III. The figures are
    analogous to Figure~\ref{fig:convtestcase2}, except that the $\LL^2$-norm
	of the error on the surface is now indicated by `global surface' and additionally a `local surface'
	$\LL^2$-norm of the error has been plotted, excluding an $0.1$-neighborhood of the dislocation.}
  \label{fig:convtestcase3}
\end{figure}

\section{Conclusions}
\label{sec:concl}

To solve Volterra's dislocation problem by standard finite-elements
techniques, the dislocation is required to coincide with
element edges. This requirement links the finite-element mesh with the
fault geometry, which prohibits the reuse of computational components 
in situations where multiple geometries have to be considered. In
particular, it renders the finite-element method infeasible in nonlinear inversion
problems.

To overcome the problems of standard finite-element techniques in nonlinear
inversion processes, in this paper we introduced the {\em Weakly-enforced Slip Method\/} (WSM),
a new finite-element approximation for Volterra's dislocation problem in which
the slip discontinuity is weakly imposed in the right-hand-side load functional.
Accordingly, the bilinear form in the formulation and, hence, the stiffness
matrix are independent of the fault geometry. The method is summarized by
the following weak formulation:
\begin{displaymath}
  u \in \VV_h:\quad a( u, v ) = -\int_\manifold \slip \cdot
  \mean{ \sigma_\nu( v ) } \quad \forall{}v\in{}\VV_h\,.
\end{displaymath}
The stiffness matrix depends on properties of the continuous
domain only, namely, variations in the elasticity and topology and geometry of the domain, 
and remains independent of fault geometry. Fault dependence manifests in the 
right-hand-side load functional only. The load functional is formed by integrating over the fault,
which is allowed to cut through elements. The integration along the fault is a non-standard
operation in finite-element methods, but we expect that it can be incorporated in most existing finite-element
toolkits with minor effort. We further note that no approximations are
required regarding fault geometry and slip distribution, unlike lift-based
methods, which require a parametrization for both. 

We established that as a consequence of the continuous approximation in WSM,
the dislocation is not resolved. Accordingly, the WSM approximation displays
suboptimal convergence in the $\LL^2$-norm under mesh refinement. In particular,
the $\LL^2$-norm of the error decays only as $O(h^{1/2})$ as the mesh width $h$ tends to~$0$,
independent of the order of approximation. Furthermore, the $\HH^1$-norm of the error 
generally diverges as $O(h^{-1/2})$ as $h\to{}0$, independent of the order of approximation.
In addition, we however proved that WSM has outstanding local approximation properties,
and that the method generally displays optimal convergence in the $\HH^1$-norm on
the domain excluding an arbitrarily small neighborhood of the dislocation. In particular,
for any $\varepsilon$-neighborhood $\varkappa^{\varepsilon}$ of the dislocation, 
the $\|\cdot\|_{1,\Omega\setminus\varkappa^{\varepsilon}}$ norm of the error in the WSM 
approximation generally converges as $O(h^p)$ as $h\to{}0$,
with~$p$ the polynomial degree of the finite-element space.

Numerical experiments in 2D and 3D were conducted to verify and scrutinize the approximation 
properties of WSM. The asymptotic error estimates for the global $\HH^1$-norm, the global $\LL^2$-norm
and the local $\HH^1$-norm of the error were confirmed in all cases, despite the fact 
that two of the test cases violate some of the conditions underlying the asymptotic error estimates.
In particular, the numerical results indicate that the error estimates extend to rupturing faults, where
the dislocation fissures the traction-free surface. The numerical experiments moreover
conveyed that the WSM approximation displays optimal local convergence in the $\LL^2$-norm,
i.e., the $\|\cdot\|_{\Omega\setminus\varkappa^{\varepsilon}}$-norm of the error decays
as $O(h^{p+1})$ as $h\to{}0$. For the non-rupturing-fault test case, the error in the surface
displacement converges optimally in the $\LL^2$-norm at $O(h^{p+1})$. For the rupturing-fault
test case, the $\LL^2$-norm of the error at the surface converges at $O(h^{1/2})$, while the
local $\LL^2$-norm excluding a small neighborhood of the dislocation again converges optimally
at $O(h^{p+1})$. Overall, the approximation obtained via WSM is very well behaved, and the
method proves remarkably robust.

Given the compelling properties of WSM one might be tempted to seek
application in other than our intended field of tectonophysics.
Dislocation plasticity comes to mind as one heavily relying on
superposition of elastic dislocations. For many problems, however, the
location of the fault will be known a-priori, in which case there is no
reason to avoid strong imposition. Moreover, often the internal stress
is an important quantity to be resolved, which with WSM suffers from
inaccuracies in regions close to the fault. In tectonophysics the
primary observable is the displacement of the free surface, which WSM is
very well capable of resolving.

We believe that WSM can play an important, if specific, role in the
application of tectonic fault plane inversions. Being able to precompute
the stiffness matrix and a quality preconditioner, for any fault geometry
and slip distribution under study, it remains only to integrate over the
2D manifold and solve the system. This makes it feasible to use finite
elements in a direct nonlinear inversion. In the hands of geophysicists,
this tool will allow all available in-situ knowledge to be made part of
the forward model. We hope this will help to improve the accuracy of
future co-seismic analyses.

\bibliographystyle{plain}
\bibliography{literature}

\end{document}